\documentclass[11pt]{article}
\usepackage{amsmath,latexsym,amsxtra,enumerate,
color, cancel}
\usepackage{bbm, dsfont}
\usepackage{enumitem}
\usepackage{amsthm}
\usepackage[english]{babel}
\usepackage[usenames,dvipsnames,svgnames,table]{xcolor}
\usepackage{times, amsfonts, amssymb, amsthm, amscd, graphicx,stmaryrd}
\usepackage[mathscr]{euscript}
\RequirePackage[colorlinks,citecolor=blue,urlcolor=blue]{hyperref}
\usepackage[active]{srcltx}
\usepackage{a4wide}

\newcommand{\la}{\lambda}

\newcommand{\R}{\mathbb{R}}
\newcommand{\N}{\mathbb{N}}
\newcommand{\E}{\mathbb{E}}
\newcommand{\PP}{\mathbb{P}}
\newcommand{\one}{\mathds{1}}
\newcommand{\tx}{{\tt x}}
\newcommand{\f}{ \mathbf{f}}
\newcommand{\Y}{\mathcal{Y}}
\newcommand{\Z}{\mathcal{Z}}
\newcommand{\U}{\mathcal{U}}
\newcommand{\om}{\omega}
\newcommand{\Om}{\Omega}
\newcommand{\ftn}{\mathcal{F}}
\newcommand{\bu}{{\bf u}}
\newcommand{\equa}{\begin{eqnarray*}}
\newcommand{\tion}{\end{eqnarray*}}
\newcommand{\equal}{\begin{eqnarray}}
\newcommand{\tionl}{\end{eqnarray}}
\newcommand{\m}{\mathbbm{m}}
\newcommand{\DD}{{\mathbb{D}_{1,2}}}
\newcommand{\les}{\hspace{-2em}}
\newcommand{\vare}{\varepsilon}
\newcommand{\al}{\alpha}
\newcommand{\ul}[1]{{\underline {#1}}}
\newcommand{\ol}[1]{{\overline {#1}}}
\newcommand{\A}{\bf A}
\newcommand{\kla}{\left ( }
\newcommand{\mer}{\right ) }
\newcommand{\non}{\nonumber }
\theoremstyle{plain}
\newtheorem{thm}{Theorem}[section]
\newtheorem{lemma}[thm]{Lemma}

\newtheorem{rem}[thm]{Remark}
\newtheorem{assumption}[thm]{Assumption}
\newtheorem{example}[thm]{Example}

\allowdisplaybreaks
\makeatletter
\def\timenow{\@tempcnta\time
\@tempcntb\@tempcnta
\divide\@tempcntb60
\ifnum10>\@tempcntb0\fi\number\@tempcntb
:\multiply\@tempcntb60
\advance\@tempcnta-\@tempcntb
\ifnum10>\@tempcnta0\fi\number\@tempcnta}
\makeatother

\newcommand{\ch}{}

\title{Existence, Uniqueness and Malliavin  Differentiability  of  L\'evy-driven BSDEs with locally Lipschitz Driver}

\author{Christel Geiss$^1$   \hspace{1.0em}   Alexander Steinicke$^2$\\
}

\date{}
\begin{document}

\maketitle
\begin{abstract}
We investigate  conditions for solvability and Malliavin  differentiability  of   backward stochastic differential equations driven by a L\'evy process.
In particular, we are interested in generators which satisfy a  local Lipschitz condition in the $Z$ and $U$ variable. This includes settings of linear, quadratic and exponential growths in those variables. 

Extending an idea of Cheridito and Nam to the jump setting and applying comparison theorems for L\'evy-driven BSDEs, we show existence, uniqueness, boundedness and Malliavin differentiability of a solution. The pivotal assumption to obtain these results is a boundedness condition on the terminal value $\xi$ and its Malliavin derivative $D\xi$. 

Furthermore, we extend existence and uniqueness theorems to cases where the generator is not even locally Lipschitz in $U.$  BSDEs of the latter type find use in exponential utility maximization.
\end{abstract}

\vspace{1em}
{\noindent \textit{Keywords:}  BSDEs with jumps; locally Lipschitz generator; quadratic BSDEs; existence and uniqueness of solutions to BSDEs;  Malliavin differentiability  of  BSDEs

}\bigskip

The paper contains  13593  words.

{\noindent
\footnotetext[1]{University of Jyvaskyla, Department of Mathematics and Statistics, P.O.Box 35, FI-40014 University of Jyvaskyla. \\ \hspace*{1.5em}
 {\tt christel.geiss{\rm@}jyu.fi}}
\footnotetext[2]{Department of Applied Mathematics and Information Technology, Montanuniversitaet Leoben, Peter Tunner-Stra\ss e \hspace*{1.5em}25/I, A-8700 Leoben, Austria. \\ \hspace*{1.5em}  {\tt alexander.steinicke{\rm@}unileoben.ac.at}}}


\section{Introduction}

In this paper, we consider existence, uniqueness and  Malliavin differentiability of one-dimensional backward stochastic differential equations (BSDEs) of the type
\equal \label{bsde00}
Y_t=\xi+\int_t^T\f(s,Y_s,Z_s,U_s)ds-\int_t^TZ_sdW_s-\int_{]t,T]\times\mathbb{R}\setminus\{0\}}U_s(x)\tilde{N}(ds,dx).
\tionl
Here  $W$  is  the  Brownian motion  and $\tilde{N}$ the  Poisson random measure associated to a L\'evy process $X$ with L\'evy measure $\nu$. 
In order to compute the Malliavin derivative of $\f,$ we require a special structure: We  assume that $\f$ can be represented by  functions $f$ and $g$, such that 
\begin{align}  \label{formf1}
\f(\om, s,y,z,{\bf u})=f\left((X_r(\om))_{r \le s}, s,y,z,  \int_{\R\setminus\{0\}}  g(s, {\bf u}(x)) \,\,(1\wedge|x|)\, \nu(dx) \right),
\end{align}
where $g(s,\cdot)$ is a locally Lipschitz continuous function on $\R$ with $g(s,0)=0.$ 
The function $f$  satisfies the following, in $(z,u)$ only   {\it local} Lipschitz condition:

There are nonnegative functions $a\in L_1([0,T]), b\in L_2([0,T])$  and  a nondecreasing, continuous function $\rho: [0, \infty)  \to [0, \infty)$ 
such that   $\forall t \in [0,T]$, $\forall \tx\in D{[0,t]}$  (the Skorohod space of c\`adl\`ag functions on  ${[0,T]}$), $(y,z,u), (\tilde{y},\tilde{z},\tilde{u})\in\R^3$:
\equa 
 |f(\tx, t,y,z,u)-f(\tx, t,\tilde{y},\tilde{z},\tilde{u})| 
 \le   a(t)|y-\tilde{y}| \!+\!  \rho(|z|\vee|\tilde{z}|\vee|u|\vee|\tilde{u}| ) \, b(t) ( |z-\tilde{z}| + |u-\tilde{u}|).
 \tion

Our first main result is Theorem \ref{bounded-sol-x} about  existence of solutions  $(Y,Z,U)$ to the BSDE \eqref{bsde00}:
If the terminal condition $\xi$ and  its Malliavin derivative  $D \xi$ are bounded, and the  Malliavin derivative  of the 
generator is bounded by a certain function depending on time and jump size, then there exists a  solution  $(Y, Z,U)$  which is Malliavin differentiable, and 
the paths of $Y, Z$ and $U$ are bounded by a constant a.s. Moreover, within a certain class of bounded processes, this solution is  unique.  \smallskip \\
Following  Cheridito and Nam \cite{CheriditoNam}, where a similar result is shown for BSDEs driven by a Brownian motion, the proof uses  a comparison theorem.
For BSDEs with jumps,  comparison theorems need an additional assumption on the generator  (see \ref{gamma} in Theorem \ref{comparison}).  
The comparison theorem provides not only a bound for $Y,$ but also  bounds for $Z$ and $U$: 
Indeed, since $Z$ and $U$ can be seen as versions of  Malliavin derivatives of $Y$ w.r.t.~the Brownian component  and  the jump component, respectively,
one can derive bounds by applying the comparison theorem to the Malliavin derivative of the BSDE. \smallskip \\
For BSDEs with  quadratic or sub-quadratic growth in $z$,  Briand  and Hu \cite{BriandHu}, Bahlali \cite{Bahlali}, S. Geiss and Ylinen \cite{GeissYlinen} 
(all in case of  BSDEs driven by a  Brownian motion)  and Antonelli and Mancini \cite{AntonelliMancini} (for BSDEs  with jumps and finite L\'evy measure),
investigate the requirements  on the terminal  condition such that existence and uniqueness of solutions holds.  It is well-known that -- in the case of quadratic growth in $z$, -- 
square integrability of $\xi$ is not  sufficient but the assumption that   $\xi$  is bounded  can be relaxed. However, for  super-quadratic drivers and 
a.s.~bounded terminal conditions $\xi$,  Delbaen et al.~\cite{delbaen} have shown that there are  cases of BSDEs without any solution as well as BSDEs with infinitely many solutions.\smallskip \\
For quadratic BSDEs with jumps  and infinite L\'evy measure there seem to be only results for bounded $\xi$  so far (see Morlais \cite{Morlais}  and Becherer et al.~\cite{BechererI}), and also for the  method we apply here, boundedness is needed. \bigskip \\
Our second main result is Theorem \ref{Hbsde}, which states   existence and uniqueness for a class of BSDEs 
where the generator is not even  locally  Lipschitz  w.r.t.  $\bu \in L_2(\nu).$ As an example, consider 
\equal \label{example-generator}
\f(s,y,z,{\bf u})=\bar {\f}(s,y,z)+\int_{\R\setminus\{0\}}\mathcal{H}_{\al}(\bu(x)) \nu(dx),
\tionl
where 
$$\mathcal{H}_{\al}(u) :=\frac{e^{\alpha u}-\alpha u-1}{\alpha},$$
for a real $\alpha>0$ and $\bar{\f}$ being quadratic in $z$. 
This particular form of $\f$ arises from exponential utility maximisation, see Morlais \cite{Morlais} or Becherer et al. \cite{BechererI}. 
Notice that compared to the generator given in \eqref{formf1}, the integral in \eqref{example-generator} does not contain the factor $1 \wedge|x|.$
In Section \ref{sec4} we address  the question to what extent the   structure of the generator  given in \eqref{example-generator} can be generalised. 
We were not able to show that the factor $1 \wedge|x|$  in  \eqref{formf1} can be  simply dropped under the given assumptions, but one can 
generalise  \eqref{example-generator} to the case where 

\equa
  \f(t,y, z, {\bf u}) :=  \varphi\left(\bar f\left(t,y, z,  G(t,{\bf u})  \right),  \int_{\R\setminus\{0\}}  \mathcal{H}({\bf u}(x))  \nu(dx) \right) 
\tion
with  $G(t,{\bf u}):=  \int_{\R\setminus\{0\}}  g(s, {\bf u}(x)) \,\,(1\wedge|x|)\, \nu(dx)$ and
$\bar f$ satisfying the assumption for \eqref{formf1}.  The function  $\varphi : \R^2 \to \R$  is a differentiable function such that
$ |\partial_v\varphi(v,w)|\leq 1$ and  $v \mapsto\partial_w\varphi(v,w)$ is a bounded function for any fixed $w \in \R.$   Moreover, we require 
$\varphi$ to satisfy  a condition  such  that the comparison theorem holds (see \ref{H3}).   The function $\mathcal{H}$ is 
a generalisation of  $\mathcal{H}_{\al}.$ It turns out that the bounds for 
$(Y,Z,U)$ do not depend on $\mathcal{H}.$

\bigskip
The paper is structured as follows:
In Section \ref{setting} we introduce  the notation and shortly recall the Skorohod  space,   Malliavin calculus for L\'evy processes
and results on existence and uniqueness of solutions to BSDEs as well as a comparison theorem for later use.
Sections \ref{sec3}   and   \ref{sec4} contain the main  results, Theorems \ref{bounded-sol-x} and \ref {Hbsde}, and their proofs. 
{\ch Section \ref{sec5} draws a connection between BSDEs with jumps and partial differential-integral equations (PDIEs).}
In  Appendix \ref{app}  we formulate a result of Malliavin differentiability for Lipschitz BSDEs which slightly  generalises  
\cite[Theorem \ref{diffthm}]{GeissSteinII}. It is applied in the proof of Theorem \ref{bounded-sol-x}.

\section{Setting and preliminaries}\label{setting}
\subsection{L\'evy process and independent random measure} \label{XandM}
Let $X=\left(X_t\right)_{t\in{[0,T]}}$ be a c\`adl\`ag L\'evy process with L\'evy measure $\nu$ on a complete probability space $(\Omega,\mathcal{F},\mathbb{P})$. 
We will denote the augmented natural filtration of $X$ by
$\left({\mathcal{F}_t}\right)_{t\in{[0,T]}}$ and assume that $\mathcal{F}=\mathcal{F}_T.$ \\
 \bigskip
 The L\'evy-It\^o decomposition of a L\'evy process $X$ can be written as
\equa
X_t = \gamma t + \sigma W_t   +  \int_{{]0,t]}\times \{ |x|\le1\}} x\tilde{N}(ds,dx) +  \int_{{]0,t]}\times \{ |x|> 1\}} x  N(ds,dx),
\tion
where $\gamma\in \R, \sigma\geq 0$, $W$ is a Brownian motion and $N$ ($\tilde N$) is the (compensated) Poisson random measure corresponding to $X$.
 The process $$\left(\int_{{]0,t]}\times \{ |x|\le1\}} x\tilde{N}(ds,dx) +  \int_{{]0,t]}\times \{ |x|> 1\}} x  N(ds,dx)\right)_{t\in[0,T]}$$
is the jump part of $X$ and will be denoted by $J$. Note that the $\mathbb{P}$-augmented filtrations $(\ftn^W_t)_{t\in{[0,T]}}$  and $(\ftn^J_t)_{t\in{[0,T]}}$ 
generated by the processes $W$  and $J,$  respectively, satisfy $$\ftn^W_t\vee\ftn^J_t=\ftn_t,$$ (see \cite[Lemma 3.1]{suv2}) thus spanning the original filtration generated by $X$ again.
Throughout the paper we will use the notation $X(\omega)=\left(X_t(\omega)\right)_{t\in{[0,T]}}$  for  sample  paths.  \\
\bigskip
 Let
\[\mu(dx):=\sigma^2\delta_0(dx)+\nu(dx)\]
and
\equa
\m(dt,dx) :=(\lambda\otimes\mu) (dt,dx),
\tion
where $\lambda$ denotes the Lebesgue measure. We define the independent random measure  (in the sense of \cite[p. 256]{ito}) $M$ by
\equal  \label{measureM}
   M(dt,dx):=\sigma dW_t\delta_0(dx) +\tilde N(dt,dx)
\tionl
 on sets $B \!\in \!\mathcal{B}([0,T]\times\R)$ with  $\m(B) < \infty$.
It holds $\E M(B)^2 = \m(B).$
In \cite{suv2}, Sol\'e et al.~consider  the independent random measure $\sigma dW_t\delta_0(dx)$ $+$ $x\tilde N(dt,dx).$ Here, in order to match the notation used for BSDEs, we work with the {\it equivalent} approach where the
Poisson random measure is not multiplied with $x$.\smallskip \\
We close this section with notation
for  c\`adl\`ag  processes on the path  space, and for BSDEs. \\ \bigskip

\subsection{Notation:  Skorohod space}
\begin{itemize}
\item With $D{[0,T]}$ we denote the Skorohod space of c\`adl\`ag functions on the interval ${[0,T]}$ equipped with the Skorohod topology. The $\sigma$-algebra
$\mathcal{B}(D{[0,T]})$  is the Borel $\sigma$-algebra  i.e.~it is generated by the open sets of $D{[0,T]}.$ It coincides with the $\sigma$-algebra generated by the family of coordinate
projections $\left(p_t\colon D{[0,T]}\to \R,\ \tx  \mapsto \tx(t),\  t\in [0,T] \right)$  (see \cite[Theorem 12.5]{Billing} for instance).
\item For a fixed $t\in{[0,T]}$ the
notation
\equa
\tx^t(s):=\tx(t\wedge s),\text{ for all } s\in{[0,T]}
\tion  
induces the natural identification $$D{[0,t]}=\left\{\tx\in D{[0,T]} : \tx^t=\tx \right\}.$$ By this identification
we define a filtration on this space by
\equal \label{filtrationG-t}
\mathcal{G}_t=\sigma\left(\mathcal{B}\left(D{[0,t]}\right)\cup \mathcal{N}_X{[0,T]}\right), \quad 0\leq t\leq T,
\tionl where $\mathcal{N}_X{[0,T]}$ denotes the
null sets of $\mathcal{B}\left(D{[0,T]}\right)$ with respect to the  image measure $\mathbb{P}_X$ on $\left(D{[0,T]},\mathcal{B}\left(D{[0,T]}\right)\right)$ of the L\'evy process $X\colon\Omega\to D{[0,T]},\omega \mapsto \mathrm{X}(\omega)$. For more details on $D{[0,T]}$, see \cite{Billing} and \cite[Section 4]{delzeith}.
\end{itemize}
\subsection{Notation for  BSDEs}
\begin{itemize}
\item Notice that $|\cdot|$ may denote the absolute value of a real number or  a norm in $\R^n.$
\item $L_p:=L_p(\Omega,\mathcal{F}, \mathbb{P}), \quad p \ge 0.$
\item  $L_p([0,T]):=L_p([0,T],\mathcal{B}([0,T]), \la), \quad p \ge 0.$
\item $L_2(\nu):= L_2(\R_0, \mathcal{B}(\R_0), \nu)$ with $\| \bu \|:= \| \bu \|_{L_2(\nu)}$ and $\R_0:= \R\!\setminus\!\{0\}$.
\item For  $1\le p \le \infty$ let  $\mathcal{S}_p$ denote the  space of all $(\mathcal{F}_t)$-progressively measurable and c\`adl\`ag processes  $Y\colon\Omega\times{[0,T]} \rightarrow \R$ such that
\equa
\left\|Y\right\|_{\mathcal{S}_p}:=\|\sup_{0\leq t\leq T} \left|Y_{t}\right| \|_{L_p} <\infty.
\tion
\item We define $L_2(W) $ as the space of all $(\mathcal{F}_t)$-progressively measurable processes $Z\colon \Omega\times{[0,T]}\rightarrow \R$  such that
\equa
\left\|Z\right\|_{L_2(W) }^2:=\E\int_0^T\left|Z_s\right|^2 ds<\infty,
\tion
and  $L_\infty(W) $ denotes the space of all $(\mathcal{F}_t)$-progressively measurable processes $Z\colon \Omega\times{[0,T]}\rightarrow \R$  such that
\equa
 \|Z \|_{L_\infty( \PP\otimes\la)} <\infty.
\tion
\item We define $L_2(\tilde N)$ as the space of all random fields $U\colon \Omega\times{[0,T]}\times{\R_0}\rightarrow \R$ 
which are measurable with respect to
$\mathcal{P}\otimes\mathcal{B}(\R_0)$ (where $\mathcal{P}$ denotes the predictable $\sigma$-algebra on $\Omega\times[0,T]$ generated
by the left-continuous $(\mathcal{F}_t)$-adapted processes) such that
\equa
\left\|U\right\|_{L_2(\tilde N) }^2:=\E\int_{{[0,T]}\times{\R_0}}\left|U_s(x)\right|^2 ds \nu(dx)<\infty,
\tion
\item  $L_{2 \times\infty}(\tilde N) $ denotes  the space of all random fields $U\colon \Omega\times{[0,T]}\times{\R_0}\rightarrow \R$
which are measurable with respect to
$\mathcal{P}\otimes\mathcal{B}(\R_0)$ 
  such that
\equa
 \left \| \int_{\R_0}\left|U_{\cdot}(x)\right|^2  \nu(dx)  \right \|_{L_\infty(\PP\otimes\la)} <\infty.
\tion
\item $L_{2,b}(\tilde N):= \{ U \in L_2(\tilde N): \exists A \in L_2(\nu)\cap L_\infty(\nu) \text{ such that } |U_s(x,\om)| \le A(x) \}.$

\item We recall the notion of the predictable projection of a stochastic process depending on parameters.
According to \cite[Proposition 3]{StrickerYor}  (see also  \cite[Proposition 3]{Meyer} or \cite[Lemma 2.2]{Ankirch}) for any {\ch non-negative or bounded}   $z\in L_2(\PP\otimes\m):=
L_2(\Omega\times {[0,T]}\times\R,\mathcal{F}_T\otimes\mathcal{B} ( [0,T]\times\R),
\mathbb{P} \otimes\m )$  there exists a process
\[^pz \in \mathrm{L}_2\left(\Omega\times {[0,T]}\times\R,\mathcal{P}\otimes\mathcal{B}(\R), \mathbb{P}\otimes\m\right)\]
such that for any fixed  $x\in\R$ the function
$ (^pz)_{\cdot,x}$ is a version of the predictable projection (in the classical sense , see e.g. \cite[Definition 2.1]{Ankirch}) of $ z_{ \cdot,x}.$ In the following we will always use this result to get predictable projections which are measurable w.r.t. a parameter. Again, we call $^pz$ the predictable projection of $z$.
\end{itemize}


\subsection{Malliavin derivatives}
We  sketch the definition of the Malliavin derivative using chaos expansions. 
For details we refer to \cite{suv}.
According to \cite{ito} there exists for any 
$\xi \in L_2(\Om,\ftn,\mathbb{P})$ a unique chaos expansion 
\equa
\xi=\sum_{n=0}^\infty I_n(\tilde f_n),
\tion
where  $f_n \in L^n_2:=L_2(([0,T]\times\R)^n,\m^{\otimes n}),$ and  $\tilde f_n((t_1,x_1),...,(t_n,x_n))$ is the symmetrisation of 
$f_n((t_1,x_1),...,(t_n,x_n))$ w.r.t. the $n$ pairs of variables. The multiple integrals $I_n$ are build with the random measure  
$M$ from \eqref{measureM}. Let $\DD$ be the space  of all  random variables $\xi \in L_2$  such that
\equa
   \|\xi\|^2_{\DD}:= \sum_{n=0}^\infty (n+1)!\left\|\tilde f_n\right\|_{L^n_2}^2<\infty.
\tion
For $\xi \in \DD,$  the Malliavin derivative is defined by
\begin{equation*}
D_{t,x}\xi:=\sum_{n=1}^\infty nI_{n-1}\left(\tilde f_n\left((t,x),\ \cdot\ \right)\right),
\end{equation*}
for $\mathbb{P}\otimes\m$-a.a. $(\omega,t,x)\in\Omega\times{[0,T]}\times\R$. It holds $D \xi  \in  L_2(\mathbb{P}\otimes\m )$.  
We will also use
\equa
\mathbb{D}_{1,2}^0:=  \bigg \{\xi=\sum_{n=0}^\infty I_n(\tilde f_n) \in L_2\colon  f_n \in  L_2^n, n\in \N,  
\sum_{n=1}^\infty  (n+1)! \int_0^T  \|\tilde f_n((t,0),\cdot) \|_{L^{n-1}_2}^2 dt < \infty  \bigg \}
\tion
and
\equa
\mathbb{D}_{1,2}^{\R_0}&:= & \bigg \{\xi=\sum_{n=0}^\infty I_n(\tilde f_n) \in L_2\colon f_n \in L_2^n, n\in \N, \\
   &&  \quad \quad \quad  \sum_{n=1}^\infty  (n+1)! \int_{[0,T]\times \R_0}  \|\tilde f_n((t,x),\cdot) 
   \|_{\mathrm{L}^{n-1}_2}^2 \m( dt,dx) < \infty  \bigg \}.
\tion

The  Malliavin derivative $D_{t,x}$ for $x\neq 0$ can be easily characterised without chaos expansions: Here we use that for
any $\xi \in  L_2(\Om,\ftn,\mathbb{P})$ there exists a measurable function $g_\xi \colon D[0,T] \to \R$ such that
 \[ \xi(\omega) =g_\xi\left(\left(X_t(\omega)\right)_{0\leq t\leq T}\right) =g_\xi(X(\omega))\] for a.a. $\omega \in \Omega$ (see, 
 for instance, \cite[ Section II.11]{bauer}).  

\begin{lemma}[\cite{Steinicke},  {\cite[Lemma 3.2]{GeissSteinII}} ] \label{functionallem}
If  $g_\xi(X) \in    L_2$
then 
\equa
 g_\xi(X) \in  \mathbb{D}_{1,2} ^{\R_0}  \iff  g_\xi(X+x\one_{[t,T]})-g_\xi(X) \in L_2(\PP \otimes\m),
\tion
and it holds then
   for    $x\neq0$  $\PP \otimes\m$-a.e.
\begin{equation}\label{xieq}
D_{t,x} \xi=  g_\xi(X+x\one_{[t,T]})-g_\xi(X) .
\end{equation}
\end{lemma}

For the canonical L\'evy space, this result can be found in \cite{Alos}.
Notice that  \cite{Alos} uses the random measure $\sigma dW_t\delta_0(dx) +x\tilde N(dt,dx)$, so that the according Malliavin derivative 
for $x\neq 0$ and $M$ from \eqref{measureM} is a  {\it difference quotient}   while we have just a  difference. However,  both approaches are 
equivalent.  \bigskip

Assume for example,  that the generator  $f\left((X_r(\om))_{r \le s}, s,y,z,u \right) $ is a.s. Lipschitz in $(y,z,u).$ 
Then also the Malliavin derivative  $ D_{t,x}f\left((X_r(\om))_{r \le s}, s,y,z,u \right)$  for $x\neq 0$ has this property for $\mathbb{P} \otimes \m$-a.a. 
$(\omega,t,x) \in \Omega\times[0,T]\times \R_0.$  This is an immediate consequence  of the next Lemma.
 
 \begin{lemma}[{\cite[Lemma 3.3]{GeissSteinII}}] \label{path-shift}
Let $\Lambda\in \mathcal{G}_T$ be a set with  $\mathbb{P}\left(\left\{X\in \Lambda\right\}\right)=0$. Then
$$\mathbb{P} \otimes\m\left(\left\{(\omega,t,x)\in \Omega\times{[0,T]}\times\R_0:X(\omega)+x\one_{[t,T]}\in \Lambda\right\}\right)=0.$$
\end{lemma}

\subsection{Existence and comparison results for monotonic generators}

We consider  the BSDE
\equal \label{bsde0}
Y_t=\xi+\int_t^T  \f (s,Y_s, Z_s, U_s)ds
 -     \int_t^T Z_s   dW_s 
-\int_{{]t,T]}\times{\R_0}}U_s(x) \tilde N(ds,dx), \,\, 0\le t\le T,
\tionl

where 
$$ \f:  \Om \times [0,T] \times \R \times \R \times L_2(\nu) \to \R.$$

If a triple $(Y,Z,U) \in \mathcal{S}_2 \times L_2(W) \times L_2(\tilde N) $ satisfies \eqref{bsde0}  it is   called a solution to the BSDE  \eqref{bsde0}.  

We will recall first  the  existence and  uniqueness  result  \cite[Theorem 3.1]{SteinickeII}.

\begin{thm}\label{existence}
There exists a unique solution to the BSDE $(\xi,\f)$ with $\xi\in L_2$ and generator $\f:\Omega\times{[0,T]}\times\mathbb{R}\times\mathbb{R}\times L_2(\nu)\to\mathbb{R}$
satisfying the properties
\begin{enumerate}[label={(H\,\arabic*)}]
\item\label{1} For all $(y,z,\bu): (\omega,s)\mapsto \f(\omega,s,y,z,\bu)$ is progressively measurable.

\item\label{2} There are nonnegative, progressively measurable processes $K_1,  K_2$ and $F$ with 
\equa 
 \left \|\int_0^T\left(K_1(\cdot ,s)+K_2(\cdot,s)^2\right)ds \right \|_{\infty}<\infty \quad \text{ and } \quad \E \left[\int_0^T |F( t)| dt\right]^2< \infty
\tion
 such that for all $(y,z,\bu)$,
\begin{align*}
&|\f(s,y,z,\bu)|\leq F(s)+K_1(s)|y|+K_2(s)(|z|+\|\bu\|), \quad \mathbb{P}\otimes\lambda\text{-a.e.}
 \end{align*}

\item\label{3} For $\lambda$-almost all $s$, the mapping $(y,z,\bu)\mapsto \f(s,y,z,\bu)$ is $\mathbb{P}$-a.s. continuous. Moreover, there is a nonnegative function $\alpha\in L^1([0,T])$, $c>0$ and a progressively measurable process $\beta$ with $\int_0^T \beta(\omega,s)^2 ds<c$, $\mathbb{P}$-a.s. such that for all $(y,z,\bu), (y',z',\bu')$,
\begin{align*}
&(y-y')(\f(s,y,z,\bu)-\f(s,y',z',\bu'))\\
&\leq \alpha(s)\theta(|y-y'|^2)+\beta(s)|y-y'|(|z-z'|+\|\bu-\bu'\|), \quad \mathbb{P}\otimes\lambda\text{-a.e.}
\end{align*}
where $\theta$ is a nondecreasing, continuous and concave function from ${[0,\infty[}$ to itself, $\theta(0)=0$, {\ch  $\limsup_{x\searrow 0}\frac{\theta(x^2)}{x}=0$} and $\int_{0^+}\frac{1}{\theta(x)}dx=\infty$.
\end{enumerate}
\end{thm}

We cite also  the comparison theorem  \cite[Theorem 3.5]{SteinickeII}.

\begin{thm}\label{comparison}
Let $\f,\f'$ be two generators satisfying the conditions \ref{1}-\ref{3} of Theorem \ref{existence} ($\f$ and $\f'$ may have different coefficients). We assume $\xi\leq \xi'$, $
\mathbb{P}$-a.s. and for all $(y,z,\bu)$, $$\f(s,y,z,\bu)\leq \f'(s,y,z,\bu),$$ for $\mathbb{P}\otimes\lambda$-a.a. $(\omega,s)\in\Omega\times{[0,T]}$.  Moreover, assume 
that $\f$ or \,$\f'$
satisfy the condition (here formulated for $\f$)
\begin{enumerate}[label={(A\,$\gamma$)}]
\item\label{gamma} \quad\quad$ \f(s,y,z,\bu)- \f(s,y,z,\bu')\leq \int_{\R_0}(\bu'(x)-\bu(x))\nu(dx), \phantom{ttt}\mathbb{P}\otimes\lambda$-a.e. \\
\hspace*{1.9em} for all $\bu, \bu'  \in L_2(\nu)$ with   $\bu\leq \bu'.$
\end{enumerate} 
Let $(Y,Z,U)$ and $(Y',Z',U')$ be the solutions to $(\xi,\f)$ and $(\xi',\f')$, respectively.
Then, $$Y_t\leq Y'_t,  \quad \mathbb{P}\text{-a.s.}$$
\end{thm}

\section{Existence  result, bounds and  Malliavin differentiability  for locally Lipschitz generators }\label{sec3}

To prove  Malliavin differentiability we restrict  ourselves to the following BSDE
\equal \label{bsde}
Y_t&=&\xi+\int_t^T  f\left( (X_r)_{r\le s}, s,Y_s, Z_s,      \int_{\R_0}  g(s, U_s(x)) \kappa (x) \nu(dx) \right)ds
 -     \int_t^T Z_s   dW_s \non \\
&&\hspace*{12em}-\int_{{]t,T]}\times{\R_0}}U_s(x) \tilde N(ds,dx), 
\tionl
where we use in the future the notation 
\equa
  G(t,\bu) :=  \int_{\R_0}  g(t, {\bf u}(x)) \kappa (x) \nu(dx), \quad {\bf u} \in L_2(\nu).
\tion
We assume  $$\kappa (x)  := 1 \wedge |x|.$$ 
\begin{rem}
To apply Malliavin calculus in the L\'evy setting, one may assume that $(\Omega,\mathcal{F},\mathbb{P})$ is  the canonical  space in the 
sense of Sol\'e et al. \cite{suv2}. On this space, since roughly speaking each $\omega \in \Omega$ represents a path of the L\'evy process 
$X = (X_t)_{t \in [0,T]}$, the Malliavin derivative $D_{t,x} \xi$ has a meaningful definition for every $\omega \in \Omega$ if  
 $\xi \in \DD.$ \\ 
Here we use a slightly different approach. We keep $(\Omega,\mathcal{F},\mathbb{P})$ as introduced in Subsection \ref{XandM} but assume any 
random object to be  a functional of $X$ so that  for  $D_{t,x}$ ($x \neq 0$) one can use Lemma  \ref{functionallem}, and  for $D_{t,0}$ we have the chain rule.
\begin{itemize}
\item
For the terminal condition $\xi$ the existence of such a functional is guaranteed by
 Doob's factorisation Lemma:  for any  $\mathcal{F}_T$-measurable $\xi  $ there exists a $g_\xi:  D{[0,T]} \to \R$ such that $\xi =  g_\xi(X)$ $\mathbb{P}$-a.s.
\item For a jointly measurable 
and adapted generator  $\f:\Omega\times{[0,T]}\times\mathbb{R}\times\mathbb{R}\times L_2(\nu)\to\mathbb{R} $   we have by \cite[Theorem 3.4]{Steinicke} that there exists 
 a jointly measurable $g_\f: D{[0,T]} \times{[0,T]}\times\mathbb{R}\times\mathbb{R}\times L_2(\nu)\to\mathbb{R}$ such that 
  $$ \f(\cdot, t,y,z,\bu)   = g_\f((X_s)_{s \in [0,T]},t,y,z,\bu) $$
  up to indistinguishability for the parameters $(t,y,z,\bu)$. Moreover, since $\f$ is adapted, for all $t$, the functional $g_\f((X_s)_{s \in [0,T]},t,\cdot,\cdot,\cdot)$ is $\mathcal{F}_t\otimes\mathcal{B}(\R)\otimes\mathcal{B}(\R)\otimes\mathcal{B}(L_2(\nu))$-measurable. Therefore, using \cite[Lemma 3.2]{Steinicke}, we may find a functional $g^t_\f: D{[0,T]} \times\mathbb{R}\times\mathbb{R}\times L_2(\nu)\to\mathbb{R}$ such that 
  $$g_\f((X_s)_{s \in [0,T]},t,\cdot,\cdot,\cdot)=g^t_\f((X_s)_{s \in [0,t]},t,\cdot,\cdot,\cdot),\quad\mathbb{P}\text{-a.s.}$$
  In other words, $g_\f$ is adapted to the filtration $\left(\mathcal{G}_t\right)_{t\in{[0,T]}}$  from \eqref{filtrationG-t}.  As $g_\f$ is adapted and measurable, there is a progressively measurable version of $g_\f$, denoted by $\bar{g_\f}$. Hence we found a progressively measurable functional to represent $\f$ in the way that 
   $$\f(\cdot, t,y,z,\bu)   = \bar{g_\f}((X_s)_{s \in [0,t]},t,y,z,\bu),\quad\mathbb{P}\text{-a.s.}$$
for all $(t,y,z,\bu)$.
\end{itemize}

\end{rem} 

The previous remark gives us the right to describe the dependency on $\omega$ through $(X_t(\omega))_{t \in [0,T]}$ in  \eqref{bsde}. For shortness of representation
we sometimes drop the dependence on $(X_t(\omega))_{t \in [0,T]}$ as it is usually done with $\omega.$

We agree on the following assumptions on $\xi$, $f$ and $g$:

\begin{assumption}
{\color{white} .} \label{assumptions}
\begin{enumerate}[label={($\mathcal{A}$\arabic*)}]
\item \label{Axi}
$A_\xi= \|\xi\|_{L^\infty(\PP)}< \infty,$ \\
 $\xi \in \DD,$  \\
 $A_{D\xi}(x) := \|(t, \om)\mapsto  D_{t,x}\xi \|_{L^\infty(\lambda \otimes \PP)} < \infty,$ \\
$\|A_{D\xi}\| < \infty.$

   \item \label{f-prime}
 for all $(y,z,u) \in \R^3$ the map
  $(\tx, t) \mapsto   f(\tx, t,y,z,u)$ is $(\mathcal{G}_t)_{t \in [0,T]}$- progressively measurable, \\
  for  all  $(\tx, t) \in D[0,T] \times [0,T], $  the functions $f ,$ $\partial_yf,  \partial_zf, \partial_uf$ are continuous in  $(y,z,u)$. 

 \item  \label{growth}
 \underline{integrability condition:} there exists a  function   $k_f   \in L_1([0,T])$ 
such that $\forall y \in \R$ 
and $\forall t \in [0,T]$  and $\forall \tx  \in D[0,T]$ it holds
\equa 
&&\hspace{-4em}    |f(\tx, t,0,0,0)|  \le   k_f(t). 
 \tion 
  \item \label{Lipschitz}
  \underline{local Lipschitz condition:} 
there exist nonnegative functions $a\in L_1([0,T]), b\in L_2([0,T])$  and  a non-decreasing continuous function $\rho: [0, \infty)  \to [0, \infty)$ 
such that $\forall t \in [0,T], \,(y,z,u), \\(\tilde{y},\tilde{z},\tilde{u})\in\R^3$ and $\forall  \tx  \in D[0,T]$ it holds 
\equa 
  |f({ \tx}, t,y,z,u)\!-\!f({\tx}, t,\tilde{y},\tilde{z},\tilde{u})| 
 \le   a(t)|y-\tilde{y}| \!+\! \rho(|z|\!\vee\!|\tilde{z}|\!\vee\!|u|\!\vee\!|\tilde{u}|)  b(t)  ( |z\!-\!\tilde{z}| \!+\! |u\!-\!\tilde{u}|).
 \tion

\item \label{A-Df}  
 \underline{Malliavin differentiability:} 
Assume that there exists  a function  $p \in L_1( [0,T],\la;  L_2(\R, \delta_0+\nu))$ 
such that if  
\equa
R &:=& A_\xi  e^{\int_0^T a(s)ds} +  \int_0^Tk_f(s) e^{\int_0^s a(r)dr}ds +1 \\
Q &:=& A_{D\xi}(0)e^{\int_0^T a(s)ds } +  \int_0^T p(s,0)  e^{\int_0^s a(r)dr}  ds +1 \\
P &:=&  \rho(2R) \|\kappa\|\Bigg (\| A_{D\xi}\|
e^{\int_0^T a(s)ds } 
  +  \int_0^T \| p(s, \cdot)\| \, e^{\int_0^s a(r)dr} ds \Bigg ) +1
\tion
and if
$rqp :=\{(y,z,u) \in \R^3:     |y| \le R, |z| \le Q, |u| \le P \}, $
then 
 \begin{enumerate}
 \item $\forall t\in [0,T], (y,z,u) \in rqp  :  f(X,t,y,z,u)  \in \DD,$
\item  for a.e. $(t,x)$  
\equa A_{Df}(t, x) :=\sup_{(y,z,u) \in rqp}\| (\omega ,s) \mapsto  (D_{s,x}  f\left(X,t,y, z,u \right))(\om) \|_{L_\infty(\PP\otimes \la)} 
\le  p(t,x).\tion
\end{enumerate}

\item   \underline{Malliavin regularity:} \label{Df-condition} 
 $\forall t\in{[0,T]}$,  
$\exists\ K^{t}\in\bigcup_{p> 1}\mathrm{L}_p$ such that  for a.a. $\omega$ and for all $(y,z,u),\\ (y', z',u') \in rqp$
\equa
 && \| \left(D_{\cdot,0}f(X, t, y,z,  u)\right)(\omega)-\left(D_{\cdot,0}f(X, t,  y',z',  u')\right)(\omega) \|_{L_2([0,T])} \\
 && \le K^{t}(\omega)(|y-y'|+|z-z'|+ |u-u'|),
\tion
\item   \label{new-g-condition}
 $g:[0,T]\times \R \to \R$ is jointly measurable
$ u \mapsto   \partial_u g(t,u) $ is continuous for all $t  \in [0,T] ,$ and  it holds  
  $$   g(t,0)=0  \quad \text{and } \quad  |\partial_u g(t,u)|  \le  \rho(|u|).$$
 
\item  \label{bounded-derivatives} 
  for all $t \in [0,T], \tx  \in D[0,T]$ and $y,z,u,u'\in \R$ it holds  $$ -1 \le \partial_u f({ \tx }, t,y,z,u) \partial_u g(t,u'). $$ 
\end{enumerate}
\end{assumption}
\bigskip

We continue with some observations and comments about these assumptions:  
\begin{rem} \rm
\begin{enumerate}[label={(\roman*)}]
\item Notice that $\| A_{D\xi} \| \le 2 A_\xi$ follows immediately from \eqref{xieq}.
\item Assumption \ref{A-Df} is trivially satisfied if there exists a $p \in L_1( [0,T],\la;  L_2(\R, \delta_0+\nu))$  such that 
$$\| (\omega ,s) \mapsto  (D_{s,x}  f\left(X,t,y, z,u \right))(\om) \|_{L_\infty(\PP\otimes \la)} \le p(t,x) $$
holds uniformly in $(y,z,u).$  
\item {\ch We need \ref{Df-condition} for the following reason:  It is not clear whether 
 the assumption  $ f(X,t,y,z,u)  \in \DD$ in \ref{A-Df} (a) together with the continuity of the partial derivatives required in \ref{f-prime}
 implies that $f(X,t,Y_t,Z_t,G(t,\bu)) \in \DD.$  Therefore, 
in \cite[Theorem 3.12]{GeissSteinII} it was shown that $G_1,G_2,G_3 \in \DD \implies f(X,t, G_1,G_2,G_3)\in \DD $
if additionally  the Malliavin regularity assumption \ref{Df-condition} holds. 
}  
\item The mean value theorem  implies that condition \ref{bounded-derivatives} is sufficient for \ref{gamma}.  
\end{enumerate}
\end{rem}
  \bigskip
The next theorem is a generalisation of Corollary  2.8 \cite{CheriditoNam} to the jump case.
For the proof of  Corollary  2.8 \cite{CheriditoNam} a comparison theorem is used  to show the boundedness of  the process $Y$ and  its 
Malliavin derivative. We will follow this idea, but for jump processes  stronger conditions are needed for comparison theorems to hold 
(see the counter example given in   \cite[Remark 2.7]{bbp}). In fact, the condition we need is \ref{gamma}.

{\ch   Malliavin differentiability of solutions to BSDEs was considered for example in   \cite{PP1},   \cite{ImkellerDelong}, \cite{ElKarouiPengQuenez}  and \cite{MaPoRe}.       
The usual procedure  - which we follow here also -  is to impose conditions on the generator and show via Picard iteration that the solution is Malliavin differentiable.
The approach in   \cite{MaPoRe} is different  since there  conditions on the  Malliavin differentiability of the  generator with the solution processes already plugged in were considered.  
}

\begin{thm}  \label{bounded-sol-x}  Assume that  \ref{Axi}  - \ref{bounded-derivatives} hold. 
Then there exists a  solution $(Y,Z,U)$ to \eqref{bsde}  which is unique in the class 
$ \mathcal{S}_\infty \times  L_\infty(W)\times (L_2(\tilde N) \cap L_{2 \times\infty}(\tilde N) ),$  
and it holds  a.s.
\equal \label{Y-estimate}
  |Y_t | &\le&  A_\xi  e^{\int_t^T a(s)ds} +  \int_t^Tk_f(s) e^{\int_t^s a(r)dr}ds  \le R-1, 
          \tionl
and for a.e. $t \in [0,T]:$
\equal \label{Z-estimate}
 |Z_t| \le  A_{D\xi}(0)e^{\int_t^T a(s)ds } +  \int_t^T  p(s,0) 
 e^{\int_t^s a(r)dr}  ds
 \leq Q -1,
 \tionl 
and for  $\la \otimes \nu$-a.e. $(t,x) \in [0,T] \times \R_0:$
\equal \label{U-estimate}
 |U_t(x)| \le  \left (A_{D\xi}(x)e^{\int_t^T a(s)ds }  +  \int_t^T  
 p(s,x)  \,e^{\int_t^s a(r)dr} ds \right ) \wedge\ (2R-2), 
 \tionl  
 which means that $U \in L_{2,b}(\tilde N).$ \\
 Moreover, it holds that $(Y,Z, U)$ is Malliavin differentiable, i.e.
\equa
Y, Z   \in L_2([0,T];\DD), \quad U\in L_2([0,T]\times\R_0;\DD),
\tion
and for $\m$- a.e. $(r,x)$  the triple $(D_{r,x}Y, D_{r,x}Z, D_{r,x}U)$  solves
\equa 
D_{r,x}Y_t& = &D_{r,x}\xi+\int_t^T F_{r,x}\left(s,D_{r,x}Y_s, D_{r,x}Z_s,D_{r,x} U_s \right)ds -\int_t^T D_{r,x}Z_sdW_s\\
&&-\int_{{]t,T]}\times{\R_0}} D_{r,x} U_{s}(v)\tilde{N}(ds,dv), \quad 0\le r \le t \le T, \non
\tion
where   $ \Theta_s :=(Y_s,Z_s, G(s,U_s))$ and
\equa
  F_{r,0} (s,  y,z,\bu)
 &:=& (D_{r,0}  f) (s, \Theta_s) + \partial_y f(s, \Theta_s)  y  + \partial_z f(s, \Theta_s) z \\
&& + \partial_u  f (s,\Theta_s) \int_{\R_0}  \partial_u   g(s,  U_s(v)) \bu(v)    \kappa (v) \nu(dv), 
\tion
and  for $x \neq 0,$
\equa
    F_{r,x} (s, y,z,\bu )&:=& (D_{r,x}  f ) (X, s, \Theta_s )   \\
 && +   f  (X+ x\one_{[r,T]}, s, \Theta_s  +  (y,z, G(s,U_s+\bu)))    - f (X+ x\one_{[r,T]},s, \Theta_s ). \\
\tion 
Setting  $D_{r,v}Y_r(\omega):= \lim_{t\searrow r}D_{r,v}Y_t(\omega) $
for all $(r,v,\omega)$ for which  $ D_{r,v}Y$ is c\`adl\`ag and $D_{r,v}Y_r(\omega):=0$ otherwise, we have

\begin{align*} {\phantom{\int}}^p\left(\left(D_{r,0}Y_r\right)_{r\in{[0,T]}}\right) &\text{ is a version of  }(Z_{r})_{r\in{[0,T]}},\\
{\phantom{\int}}^p\left(\left(D_{r,v}Y_r\right)_{r\in{[0,T]}, v\in\R_0}\right) &\text{ is a version of  }(U_{r}(v))_{r\in{[0,T]}, v\in\R_0}.
\end{align*}
\end{thm}

(The definition of the  objects $(D_{r,x} f) (s, \Theta_s )$,  where we first apply the Malliavin derivative to 
$f$ and afterwards insert the expressions in $\Theta_s$, indeed constitute well defined measurable objects because of the continuity assumptions 
on $\f$, as is explained in \cite[Lemma 3.5 ff]{GeissSteinII} or \cite[Remark 5.3 (ii)]{Steinicke}.)

\begin{proof} 
{\bf Step 1}
For $M \in \R_+$ let  { $b_M :\R \to [-M,M]$} be a smooth monotone function such that  $0\le b'_M(x) \le 1$ and
\equa
b_M(x) := \left \{ \begin{array}{cl} M, &\,\, x>M+1,\\
x,  & \,\, |x| \le  M -1, \\
-M, &\,\, x< -M-1.
\end{array} \right .
\tion
Notice that $|b_M(x)|\le |x|$.
We set  (using $R,Q$ and $P$ from \ref{A-Df})
\equa
\widehat f(s,y,z, G(s, \bu)):=  f(s,b_R(y), b_Q(z),  \widehat G(s, \bu)),
\tion
where $$ \widehat G(s, \bu) := b_{P}(  G(s, b_{2R}(\bu)))  \quad \text{ and } \quad   b_{2R}(\bu)(x) := b_{2R}(\bu(x)).$$

We will show first that $ \widehat \f(s,y,z,\bu):=\widehat f(s,y,z, G(s, \bu))$ satisfies the  assumptions of Theorem \ref{existence}.
We have  \ref{1} because of \ref{f-prime}  and \ref{new-g-condition}; while  \ref{3} follows from \ref{Lipschitz} and \ref{new-g-condition}. Indeed, since
\ref{new-g-condition} implies 
\equa
 |\widehat G(s, \bu)- \widehat G(s, \bu')|   &\le&\sup_{u \in [-2R,2R]} |\partial_u g(s,u)|\int_{\R_0} |\bu(x)-\bu'(x)|\kappa(x)\nu(dx) \\
&\le& \rho(2R) \|\kappa \|\|\bu -\bu'\|,
\tion
it holds 
\equal \label{Lipschitz-estimate}
&& \hspace{-3em} |f(\tx, s,b_R(y), b_Q(z), \widehat G(s, \bu))  -f(\tx, s,b_R(y'), b_Q(z'), \widehat G(s, \bu')) | \notag \\
&\le& a(s)  |y-y'| + \rho({Q \vee P})b(s) (|z-z'| +  |\widehat G(s, \bu)- \widehat G(s, \bu')|) \notag \\
&\le&  a(s)  |y-y'| + \rho({Q \vee P}) (1+\rho(2R) \|\kappa \|)  b(s) (|z-z'| +  \|\bu -\bu'\| ).
\tionl
Now we combine the last inequality for $y'=0, z'=0,$ and $\bu' =0$ with \ref{growth} to get \ref{2}:
\equa
 |f(\tx, s,b_R(y), b_Q(z),  \widehat G(s, \bu))| 
\le k_f(s) +a(s)|y| +  \rho({Q \vee \!P}) (1+\rho(2R) \|\kappa \|) b(s)  (|z|  +  \|\bu\| ).
\tion
Hence by Theorem \ref{existence} there exists
 for any $\xi \in L_2$   a unique solution   $(\widehat Y, \widehat Z, \widehat U)$ to  \eqref{bsde} with data $(\widehat \f, \xi).$ \\
 
 Assumption \ref{bounded-derivatives} implies that $\widehat \f$ satisfies  \ref{gamma} from Theorem \ref{comparison}. 
 \bigskip

{\bf Step 2}
From   \ref{growth}  and \eqref{Lipschitz-estimate} we conclude that  $\forall s \in [0,T]$ and $\forall (y,z,\bu) $
it holds
\equa
|f(\tx, s,b_R(y), b_Q(z),  \widehat G(s, \bu))| 
&\le& k_f(s) +a(s)|y| +  b(s) \rho( Q\vee P) (|z|  +   | \widehat G(s,\bu)|).
\tion
We want to apply  the comparison theorem to the  BSDEs:
\equal 
\ol Y_t &=& A_{\xi} + \int_t^T k_f(s) +a(s)|\ol Y_s| + b(s)\rho(Q\vee P)(|\ol Z_s|  +   |\widehat G(s,\ol U_s)|)ds  \non \\
&&- \int_t^T \ol Z_s dW_s    - \int_{]t,T] \times{\R_0}} \ol U_s(x) \tilde N(ds,dx), \label{olY}   \\
\widehat Y_t &=& \xi + \int_t^T\widehat f(s, \widehat Y_s, \widehat Z_s , G(s, \widehat U_s))ds  \non \\
&&- \int_t^T \widehat  Z_s dW_s    - \int_{]t,T] \times{\R_0}}  \widehat  U_s(x) \tilde N(ds,dx), \non \\
\ul Y_t &=& -A_{\xi} - \int_t^T  k_f(s) +a(s)|\ul Y_s|  +b(s)\rho( Q\vee P)(|\ul Z_s|  +   |\widehat G(s,\ul U_s)|)ds  \non \\
&&- \int_t^T \ul Z_s dW_s    - \int_{]t,T] \times{\R_0}} \ul U_s(x) \tilde N(ds,dx). \label{ulY}
\tionl
By Step 1 the generator $\widehat f$ satisfies the conditions  \ref{1}-\ref{3} and \ref{gamma}.  
Since also  the generators of $\ol Y$ and $\ul Y$  satisfy the conditions \ref{1}-\ref{3}, Theorem \ref{comparison} implies that

\equa
\ul Y_t \le \widehat Y_t \le \ol Y_t, \quad \forall t \in [0,T]  \,\,\, \PP \text{-a.s.}
\tion
By Theorem \ref{existence} we have that $(\ol Y, 0,0)$ and $(\ul Y, 0,0)$ are the unique solutions to \eqref{olY} and \eqref{ulY}, respectively, and 
\equa
\ol Y_t = -\ul Y_t&=& A_\xi  e^{\int_t^T a(s)ds} +  \int_t^Tk_f(s) e^{\int_t^s a(r)dr}ds \\
&\le & A_\xi  e^{\int_0^T a(s)ds} +  \int_0^Tk_f(s) e^{\int_0^s a(r)dr}ds =R-1, 
\tion
where $R$ was defined in \ref{A-Df}. This gives \eqref{Y-estimate}   for  $\widehat Y.$

{\bf Step 3}  To consider Malliavin derivatives we check the conditions of Theorem \ref{diffthm}  for the BSDE with data $(\widehat f, \xi).$ Condition  \ref{Axi} 
implies that ($\A_{\xi}$) is satisfied. Condition (${\A}_f$) \ref{f-meas}
follows from \ref{f-prime}.   Condition \ref{growth} implies  (${\A}_f$) \ref{f(000)-Ltwo}. 

The Lipschitz  continuity required in (${\A}_f$) \ref{f-Lip} is fulfilled because of 
\ref{Lipschitz}. Furthermore, we have the implications 
\ref{A-Df} $\implies$ (${\A}_f$) \ref{f-Mall-diff-bounded},
 \ref{Df-condition} $\implies$ (${\A}_f$) \ref{f-Mall-diff-Lip}
and \ref{new-g-condition} $\implies$ (${\A}_f$) \ref{g-condition}.

Consequently, we may consider the Malliavin derivative of the BSDE \eqref{bsde} with data $(\widehat f, \xi).$ 

Let $ \Theta_s = (\widehat Y_s, \widehat Z_s, \widehat G(s,\widehat U_s)).$
Then
\equal \label{derivative-0}
 D_{r,0} \widehat Y_t
&=& D_{r,0} \xi \,+\int_t^T \bigg[ (D_{r,0}  \widehat f) (s,  \Theta_s)  + \partial_y \widehat f (s,\Theta_s) D_{r,0}  \widehat Y_s  
  + \partial_z  \widehat f (\Theta_s) D_{r,0}  \widehat Z_s \non  \\
&&  + \partial_u \widehat f (s,\Theta_s) b_P'(G(s, b_{2R}(\widehat U_s))) \non \\
&&  \quad \quad \quad \quad \times\int_{\R_0}  \partial_u   g(s,  b_{2R}( \widehat U_s(v)))b_{2R}'(\widehat U_s(v)) D_{r,0} \widehat U_s(v)    \kappa (v) \nu(dv)\bigg]  ds
 \non \\
&&       -     \int_t^T D_{r,0}  \widehat Z_s   dW_s        -\int_{{]t,T]}\times{\R_0}}D_{r,0} \widehat U_s(v) \tilde N(ds,dv).
\tionl
By \ref{A-Df} we have $| (D_{r,0}  \widehat f) (s,  \Theta_s)| \le p(s,0),$ and the Lipschitz coefficients  from \eqref{Lipschitz-estimate} are bounds for the partial
derivatives, so that by Theorem \ref{comparison}, the solutions of  
\equa
 \ol{\Y^{r,0}_t} = &&\hspace{-1em} A_{D\xi}(0) \,+ \int_t^T \bigg[ p(s,0) + a(s) \, | \ol{\Y^{r,0}_s}|
  + \rho({Q \vee P}) b(s) \, |\ol{ \Z^{r,0}_s} | \non \\
  && \hspace*{2em}+  \rho({Q \vee P}) b(s)\rho(2R) \int_{\R_0}  | \ol{ \U^{r,0}_s(v)} | \,  \kappa (v) \nu(dv) \bigg] ds
 \non \\
&&\hspace*{1em}        -     \int_t^T  \ol{\Z^{r,0}_s} dW_s        -\int_{{]t,T]}\times{\R_0}} \ol{\U^{r,0}_s(v)} \tilde N(ds,dv),
\tion
and
\equa
\ul{ \Y^{r,0}_t} =&&\hspace{-1em} -A_{D\xi}(0) \,- \int_t^T\bigg[  p(s,0) + a(s) \, |\ul{\Y^{r,0}_s}|
  + \rho({Q \vee P}) b(s)  \,  |\ul{\Z^{r,0}_s}|   \non \\
  && \hspace*{1em}+  \rho({Q \vee P}) b(s)\rho(2R) \int_{\R_0}  | \ul{\U^{r,0}_s(v)} | \, \kappa (v) \nu(dv) \bigg] ds
 \non \\
&&\hspace*{0.5em}        -     \int_t^T \ul{\Z^{r,0}_s}   dW_s        -\int_{{]t,T]}\times{\R_0}} \ul{\U^{r,0}_s(v)} \tilde N(ds,dv), 
\tion
satisfy $\ul \Y^{r,0}_t  \le D_{r,0} \widehat Y_t  \le \ol \Y^{r,0}_t.$ Note that condition \ref{gamma} required in Theorem \ref{comparison} is satisfied by the linear 
generator of equation \eqref{derivative-0} using assumption \ref{bounded-derivatives}.
The above equations do have  unique solutions where $\ol \Y$ and $\ul \Y$ are given by
\equa
\ol{\Y^{r,0}_t} = - \ul{ \Y^{r,0}_t}  =   A_{D\xi}(0)e^{\int_t^T a(s)ds } +  \int_t^T p(s,0) e^{\int_t^s a(r)dr}  ds\le Q-1.
\tion
According to Theorem \ref{diffthm} \ref{dalet} we have that  $D_{r,0}\widehat Y_r(\omega):= \lim_{t\searrow r}D_{r,0}\widehat Y_t(\omega).$
{\ch Since  $D_{r,0}\widehat Y_r(\omega)$ is measurable and bounded,  its predictable projection exists  and} is a version of  $(Z_r)_{r\in [0,T]}$  which proves \eqref{Z-estimate} for the BSDE with  $(\widehat f, \xi)$.
For $x \neq 0$ we get a similar result:
 \equa
D_{r,x} \widehat Y_t &=& D_{r,x} \xi \,+\int_t^T\bigg[   (D_{r,x}  \widehat f) (s,  \Theta_s)  + 
 \widehat {f} (X+ x\one_{[r,T]}, s, \Theta_s  +  D_{r,x} \Theta_s)  \notag \\
 &&   - \widehat{f}  (X+ x\one_{[r,T]},s, \Theta_s )
  \bigg]   ds  
     -     \int_t^T D_{r,x}  \widehat Z_s   dW_s        -\int_{{]t,T]}\times{\R_0}}D_{r,x} \widehat U_s(v) \tilde N(ds,dv),  \notag \\
 \ol{\Y^{r,x}_t}& = & A_{D\xi}(x) \,+ \int_t^T\bigg[  p(s,x) + a(s) \, |\ol{\Y^{r,x}_s}|
  + \rho({Q \vee P}) b(s)   \, | \ol{ \Z^{r,x}_s} | 
  \non \\   
 && +   \rho({Q \vee P}) b(s)\rho(2R) \int_{\R_0}  | \ol{ \U^{r,x}_s(v)} | \,  \kappa (v) \nu(dv) \bigg] ds
 \non \\
&&     -     \int_t^T  \ol{\Z^{r,x}_s} dW_s        -\int_{{]t,T]}\times{\R_0}} \ol{\U^{r,x}_s(v)} \tilde N(ds,dv),
\tion
and
\equa
\ul{ \Y^{r,x}_t} &=& -A_{D\xi}(x) \, - \int_t^T\bigg[  p(s,x) + a(s)  \, |\ul{\Y^{r,x}_s}|
  + \rho({Q \vee P}) b(s)  \, | \ul{\Z^{r,x}_s} |  \non \\
  && +  \rho({Q \vee P}) b(s)\rho(2R)  \int_{\R_0}   |\ul{\U^{r,x}_s(v)} | \, \kappa (v) \nu(dv) \bigg] ds
 \non \\
&&        -     \int_t^T \ul{\Z^{r,x}_s}   dW_s        -\int_{{]t,T]}\times{\R_0}} \ul{\U^{r,x}_s(v)} \tilde N(ds,dv), 
\tion
We get $\ul \Y^{r,x}_t  \le D_{r,x} \widehat Y_t  \le \ol \Y^{r,x}_t,    $ where
\equa
\ol{\Y^{r,x}_t} = - \ul{ \Y^{r,x}_t}  =  A_{D\xi}(x)e^{\int_t^T a(s)ds }  +  \int_t^T  p(s,x)e^{\int_t^s a(r)dr} ds.
\tion
 Additionally, notice that according to  Theorem \ref{diffthm} \ref{dalet}, $\widehat U_t(x)$ can be expressed $\mathbb{P}\otimes\m$-a.e. as 
 $\lim_{r\searrow t}D_{t,x} \widehat Y_r$ for $ x\neq 0.$ 
 Thus, by the representation \eqref{xieq} of $D$ as difference operator in this case, we end this step by stating 
$$|\widehat U_t(x)|\leq 2\sup_{0\leq t \leq T} |\widehat Y_t|\leq  2(R-1),\quad \mathbb{P}\otimes\m\text{-a.e.},$$
 which implies the estimate \eqref{U-estimate} for $\widehat U$. {\ch Again, since $\widehat U_t(x) =\lim_{r\searrow t}D_{t,x} \widehat Y_r$ is measurable and bounded, we can take its predictable projection.}
Moreover,  by the definition of $P$ in \ref{A-Df},
\equa
 \| \widehat U_t \| &\le&  \|  A_{D\xi} \|e^{\int_0^T a(s)ds }  
                 +  \int_0^T \| p(s,\cdot )\| e^{\int_0^s a(r)dr} ds = \frac{P-1}{\rho(2R) \| \kappa \|}, \quad\text{so that } \widehat U \in  L_{2,b}(\tilde N).
\tion
  
 {\bf Step 4} The assertion is now shown for the generator $\widehat \f$, thus the goal of this step is to obtain the results also for 
 $\f$ without any cut-off restraints.
By Step 3 we get $|\widehat Y_t|\leq R -1 ,|\widehat Z_t|\leq Q-1$, a.s., $|\widehat U_t(x)|\leq 2R-2$ $\mathbb{P\otimes\m}$-a.e. 
and 
\equa
 | G(s,\widehat U_s)| 
&=&   \left |  \int_{\R_0}  g(s, \widehat U_s(x)) \kappa (x) \nu(dx) \right | \notag \\
&\le & \rho(2R) \int_{\R_0} |\widehat U_s(x) | \kappa (x) \nu(dx) \notag  \\
&\le & \rho(2R) \| \kappa \| \,  \|\widehat U_s \| \le P-1.  
\tion

In that case, the equality 
$$\widehat f(t,\widehat Y_t,\widehat Z_t, G(t,\widehat U_t))= f(t,\widehat Y_t,\widehat Z_t, G(t,\widehat U_t))$$ 
holds almost surely for every $t\in{[0,T]}$. Therefore, the solution $(\widehat Y,\widehat Z, \widehat U)$ to the BSDE with data 
$(\widehat \f, \xi)$ also serves as solution of the equation given by $(\f, \xi)$, which is \eqref{bsde}.
  
{\bf Step 5} Finally we show the uniqueness of solutions in the space  $ \mathcal{S}_\infty \times  L_\infty(W)\times (L_2(\tilde N) \cap L_{2 \times\infty}(\tilde N) ).$  
Assume that  we have two solutions $(Y^{(j)},Z^{(j)} , U^{(j)}  )$ ($j=1,2$)  with  $ \sup |Y^{(j)}_t| \le R_j-1 $ and  $ \sup |Z^{(j)}_t| \le  Q_j-1$  
and ${\rm ess\, sup}_{s, \om}  \| U^{(j)}_s \| \le C_j.$
Then   $ \sup |U^{(j)}_t(x)| \le 2R_j-2 ,$ and  by \ref{new-g-condition}

\equa
 |G(s,U^{(j)}_s)| 
&=&   \left |  \int_{\R_0}  g(t, U^{(j)}_s(x)) \kappa (x) \nu(dx) \right |\\
&\le & \rho(2R_j-2) \int_{\R_0} | U^{(j)}_s(x) | \kappa (x) \nu(dx) \\
&\le & \rho(2R_j-2) \| \kappa \| \, {\rm ess\, sup}_{s, \om}  \| U^{(j)}_s \|   \\
&\le & \rho(2R_j-2) \| \kappa \|C_j =:M_j .
\tion

Hence from \ref{Lipschitz} and   it follows that 
\equa 
&& \les |\f({ \om}, t,Y_t^{(1)},Z_t^{(1)},U^{(1)}_t)-\f({ \om}, t,Y_t^{(2)},Z_t^{(2)},U^{(2)}_t)|  \\
 &\le& a(t)|Y_t^1-Y_t^2| +b(t) \rho(Q_1\vee Q_2\vee M_1\vee M_2 )
\, ( |Z_t^1-Z_t^2 | + |G(t,U^{(1)}_t)-G(t,U^{(2)}_t)|) 
\tion
And by \ref{new-g-condition}
$$
|G(t,U^{(1)}_t)-G(t,U^{(2)}_t)| \le \rho(|2R_1|\vee|2R_2|)  \, \| \kappa \| \, \| U^{(1)}_t - U^{(2)}_t\|.
$$
 Because  the processes $(Y^{(j)},Z^{(j)} , U^{(j)})$ are bounded and $f$ is locally Lipschitz,  $f$  restricted to a bounded set 
is Lipschitz,  and   uniqueness  follows from Theorem \ref{existence}. 
\end{proof}

\begin{rem}
By a mollifying argument, to weaken differentiability conditions on the generator, assumptions \ref{f-prime} and \ref{new-g-condition} may be relaxed to\begin{enumerate}[label={(A\arabic*')}]
\setcounter{enumi}{1}
 \item 
 for all $(y,z,u) \in \R^3$ the map
  $(\tx, t) \mapsto   f(\tx, t,y,z,u)$ is $(\mathcal{G}_t)_{t \in [0,T]}$- progressively measurable, \\
  for  all  $(\tx, t) \in D[0,T] \times [0,T]$ and the function $f $ is continuous in  $(y,z,u)$. 
\setcounter{enumi}{6}
\item   
 $g:[0,T]\times \R \to \R$ is jointly measurable and for all $t  \in [0,T]$ it holds  $g(t,0)=0$  and 
  $$ \forall u,u'\in \R: |g(t,u)-g(t,u')|  \le  \rho(|u|\vee|u'|)|u-u'|.$$
\end{enumerate}
However, instead of \ref{A-Df} we have to impose the slightly stronger condition
\begin{enumerate}[label={(A\arabic*')}]
\setcounter{enumi}{4}
\item  \label{A-Df'}  
 Assume that there exists $\varepsilon>0$ and a function  $p \in L_1( [0,T],\la;  L_2(\R, \delta_0+\nu))$ 
such that if  
$$rqp_\varepsilon :=\{(y,z,u) \in \R^3:     |y| \le R+\varepsilon, |z| \le Q, |u| \le P+\varepsilon \},$$ with $R,Q,P$ from \ref{A-Df},
then 
 \begin{enumerate}
\item $\forall t\in [0,T], (y,z,u) \in rqp_\varepsilon  :   f(X,t,y,z,u)  \in \DD,$
\item  for a.e. $(t,x)$  
\equa A_{Df}(t, x) &:=&\sup_{(y,z,u) \in rqp_\varepsilon}\| (\omega ,s) \mapsto  D_{s,x}  f\left(X(\om),t,y, z,u \right) \|_{L_\infty(\PP\otimes \la)} \\
&\le&  p(t,x).\tion
\end{enumerate}
\end{enumerate}
 As was the case for \ref{A-Df}, if there exists a $p \in L_1( [0,T],\la;  L_2(\R, \delta_0+\nu))$  such that 
$$\| (\omega ,s) \mapsto  D_{s,x}  f\left(X(\om),t,y, z,u \right) \|_{L_\infty(\PP\otimes \la)} \le p(t,x) $$
holds uniformly in $(y,z,u),$ then \ref{A-Df'} is trivially satisfied. \bigskip
Condition \ref{bounded-derivatives} becomes
\begin{enumerate}[label={(A\arabic*')}]
\setcounter{enumi}{7}
\item  for all $t \in [0,T],  \tx  \in D[0,T]$ and $y,z\in \R$, the generalised function (in the sense of distributions on $\mathbb{R}^2$ and using weak derivatives) $$\left(\partial_u f({ \tx }, t,y,z,\cdot) \otimes \partial_{u_1} g(t,\cdot)\right)+1$$ 
is nonnegative.
\end{enumerate}
\end{rem}

\section{A generalisation of the local Lipschitz condition}\label{sec4}

In this section we address the question whether one may  remove    the factor  $\kappa(x)= 1\wedge|x|$  in  
$$G(t,\bu) =  \int_{\R_0}  g(t, {\bf u}(x)) \,(1\wedge|x|)\, \nu(dx)$$
of  the generator $f(t,y,z,G(t,\bu)).$
For this, one could  replace $\kappa(x)$ by  $\kappa_n(x) := 1\wedge|nx|$ and let $n\to \infty.$    Notice that  $\kappa_n(x) \to 1$  for all $x \in \R_0.$
If we consider for example for some $\al >0$ the expression
$$G_{\al,n}(t,\bu):=\int_{\R\setminus\{0\}}\mathcal{H}_{\al}(\bu(x)) \kappa_n(x) \nu(dx) \quad \text{with}  \quad \mathcal{H}_\al(u) = \frac{e^{\al u} -\al u -1}{\al}, $$
then   $|G_{\al,n}(t,\bu)| \le   \frac{e^{\al \|u\|_\infty}}{\al}  \int_{\R\setminus\{0\}} |\bu(x)|^2 \nu(dx)$ for all $n\in \N,$   
so that it seems possible to consider the limit $n\to \infty$ for $ \bu \in L^\infty(\nu) \cap L^2(\nu).$

However, in condition \ref{A-Df} the factor $\| \kappa\|$ appears for the constant $P,$ and since $\| \kappa_n\| \to \infty$ if $\nu(\R_0)=\infty,$ this would lead to $P = \infty.$ \bigskip

Nevertheless, generators including the case  $f(t,y,z)+ \int_{\R\setminus\{0\}}\mathcal{H}_{\al}(\bu(x))\nu(dx)$ have been treated in \cite{BechererI} (see also  \cite{Morlais}). \bigskip

We will consider the following situation: 

Let $\bar{f}$ be a generator satisfying assumptions \ref{Axi}-\ref{bounded-derivatives} (so that Theorem \ref{bounded-sol-x} applies) and define 
\equal \label{new-f}
   \f(t,y, z, {\bf u}) :=  \varphi\left(\bar f\left(t,y, z,  G(t,{\bf u})  \right),  \int_{\R_0}  \mathcal{H}({\bf u}(x))  \nu(dx) \right)
\tionl
with
\equal\label{old-G}
  G(t,{\bf u}) :=  \int_{\R_0}  g(t, {\bf u}(x)) \kappa (x) \nu(dx), \quad \text{ for } \,\, {\bf u} \in L_2(\nu).
\tionl
For the functions $\bar f$, $\mathcal{H}:\R\to\R$ and $\varphi:\R^2\to\R$ we require the following conditions:\smallskip 

\begin{assumption}[A$\mathcal{H}$]\label{assumptionsH} $ $\\
Suppose that $\bar{f}$  satisfies  \ref{Axi}-\ref{new-g-condition} and assume that  $\f$ is given by \eqref{new-f} and \eqref{old-G}.
\begin{enumerate}[label={($\mathcal{H}$\arabic*)}]
\item\label{H1}
Let $\mathcal{H}:\R\to\R$  be such that $\mathcal{H}(0)=0$.
\item\label{H2}  We assume that     $\mathcal{H}:\R\to\R$ and $\varphi:\R^2\to\R$ are continuously differentiable, and the following conditions hold:
$$\forall v,w\in\R: |\partial_v\varphi(v,w)|\leq 1,$$
$$\forall w\in\R: v\mapsto\partial_w\varphi(v,w)\quad\text{is a bounded function.}$$
\item\label{H3}  Instead of \ref{bounded-derivatives}, we impose that for all $t \in [0,T], \tx  \in D[0,T]$ and $w,y,z,u,u'\in \R$ it holds  
$$ -1 \le \partial_v\varphi\left(\bar f\left( \tx, t,y, z,  u  \right),  w \right)\partial_u \bar f( \tx,t,y,z,u) \partial_u g(t,u') $$ and
$$-1 \le \partial_v\varphi\left(\bar f\left( \tx, t,y, z,  u  \right),  w \right)\partial_u \bar f( \tx,t,y,z,u) \partial_u g(t,u')+\partial_w\varphi\left(\bar f\left( \tx, t,y, z,  u  \right),  w \right)\mathcal{H}'(u').$$
\item\label{H4} For any $R'>0$, there is a constant $c_{R'}$ such that $|\mathcal{H}'(u)|\leq c_{R'}|u|$ for all $|u|\le R'.$ 
\end{enumerate}
\end{assumption} 
Note that the  generator $\f$ satisfying (A$\mathcal{H}$)   is not locally Lipschitz in {\bf u} $\in L_2(\nu)$.

\begin{thm}\label{Hbsde}
Under Assumption \ref{assumptionsH} (A$\mathcal{H}$) 
(with notation taken from Assumption \ref{assumptions}) there exists a solution to \eqref{bsde} if $f$ is replaced by \eqref{new-f}. 
The solution processes $(Y,Z,U)$ of this equation have the same bounds which Theorem \ref{bounded-sol-x} states for the solution of the BSDE given by $(\xi,\bar{f})$.
 The solution $(Y,Z,U)$ is unique in the class $ \mathcal{S}_\infty \times  L_\infty(W)\times L_{2,b}(\tilde N) .$ 
\end{thm}

\begin{proof}

We define for $n\in\N$, $$\mathcal{H}^n(u,x) := \mathcal{H}(u) \min\{1,n |x| \} $$
and 
$$ \f^n(s,y, z, {\bf u}) :=  \varphi\left(\bar f\left(s,y, z,  G(s,{\bf u})  \right),  \int_{\R_0}  \mathcal{H}^n({\bf u}(x),x)  \nu(dx)\right).$$  

{\bf Step 1}

 For  $n\in\N$  let   $(Y^n,Z^n,U^n)$ be the unique solution to  $(\xi, \f^n)$ which exists, since by the conditions in 
 Assumption \ref{assumptionsH} (A$\mathcal{H}$)  also Assumption \ref{assumptions} is met. Like in Step 4 of Theorem \ref{bounded-sol-x} we see that for $\bar f$ and $G$ there are Lipschitz functions $\hat{\bar{f}}, \hat{G}$  (in the sense of \ref{Lipschitz} and \ref{new-g-condition} with constant $\rho$) which, if the solution processes are inserted, give the same values as $\bar{f}$ and $G$, $\mathbb{P}\otimes\lambda\otimes\nu$-a.e. 

Moreover, Theorem  \ref{bounded-sol-x} implies that  for all $n\in  \N$
\equa
 |U^n_t(x)| &\le&    A_{D\xi}(x) e^{\int_t^T a(s)ds} + \int_t^T p(s,x) e^{\int_t^s a(r)dr}  ds. \tion
 We also know that    $|U^n_t(x)| \le   R'$ for $\mathbb{P}\otimes\m$-a.a. $(\omega,t,x)$, where $R' = 2R-2$  is the constant bound 
 for $U^n$  appearing in \eqref{U-estimate}.
  The $\mathcal{H}^n$ are a deterministic 
 functions, and therefore do not contribute to the integral term $\int_t^T p(s,x) e^{\int_t^s a(r)dr} ds$ which bounds the size of $U^n_t(x)$. By  \ref{H1} and \ref{H4}, 
 on $\{u: |u| \le   R'\}$ there is $c_{R'}>0$ such that $  |\mathcal{H}(u)| \le{ c_{R'} u^2}$.
  Therefore, we observe by the use of \ref{H2} that
  \begin{align}\label{initial}
  &\left|\f^n(s,Y^n_s, Z^n_s, U^n_s)-\f^n(s,Y^m_s, Z^m_s, U^m_s)\right|\non \\
  =&\left|\varphi\left(\hat{\bar f}(s,Y_s^n,Z^n_s,  \hat{G}(s,U^n_s)), \int_{\R_0}\!  \mathcal{H}^n(U_s^n(x),x)  \nu(dx)\right)\right.\non \\
&\left.\quad\quad\quad-   \varphi\left(\hat{\bar f}(s,Y_s^m,Z^m_s,  \hat{G}(s,U^m_s)),\int_{\R_0}\!  \mathcal{H}^n(U_s^m(x),x)  \nu(dx)\right) \right|\non \\
\leq & \left|\hat{\bar f}(s,Y_s^n,Z^n_s,  \hat{G}(s,U^n_s))-\hat{\bar f}(s,Y_s^m,Z^m_s,  \hat{G}(s,U^m_s))\right|\\
&+\left|\partial_w\varphi\left(\hat{\bar f}(s,Y_s^m,Z^m_s,  \hat{G}(s,U^m_s)),\vartheta\right)\right| \non\\
& \quad\quad \times \bigg ( \int_{\R_0} \left|\mathcal{H}(U_s^n(x)) \min\{1,n |x| \}  - \mathcal{H}(U_s^m(x)) \min\{1,n |x| \} \right|  \nu(dx)  \bigg  ),\non
  \end{align}
  where \begin{align*}
& \les  \les \min\left\{\int_{\R_0}\!  \mathcal{H}^n(U_s^n(x),x)  \nu(dx),\int_{\R_0}\!  \mathcal{H}^n(U_s^m(x),x)  \nu(dx)\right\}\leq\vartheta\\
&\leq \max\left\{\int_{\R_0}\!  \mathcal{H}^n(U_s^n(x),x)  \nu(dx),\int_{\R_0}\!  \mathcal{H}^{n}(U_s^m(x),x)  \nu(dx)\right\}.\end{align*}

Condition \ref{H4} and the bounds for $U^n$ and $U^m$ imply that
\begin{align*}
&\int_{\R_0} \max \{ |\mathcal{H}(U_s^n(x))|,|\mathcal{H}(U_s^m(x))| \}\nu(dx)   \\
&\leq \int_{\R_0} \bigg |  c_{R'}   \left ( A_{D\xi}(x) e^{\int_s^T a(v)dv}+  \int_s^T p(r,x) e^{\int_s^r a(v)dv} dr\!\right )^2\bigg |     \nu(dx)\\
&\leq C_1\left (\|A_{D\xi}\|^2+\int_0^T\| p(r,\cdot)\|^2dr\right)   <\infty 
\end{align*} 
 for a constant $C_1= C_1\Big(c_{R'}, \int_0^T a(v)dv \Big) $.
Therefore, because $\vartheta$ is bounded and by \ref{H2}, there is a constant $K$ such that
$\left|\partial_w\varphi\left(\bar f(s,Y_s^m,Z^m_s,  G(s,U^m_s)),\vartheta\right)\right|\leq K.$ 

The estimate \eqref{initial} can then be continued to
 \begin{align}\label{initial2}
  &\left|\f^n(s,Y^n_s, Z^n_s, U^n_s)-\f^n(s,Y^m_s, Z^m_s, U^m_s)\right|\non \\
  \leq & \left|\hat{\bar f}(s,Y_s^n,Z^n_s,  \hat{G}(s,U^n_s))-\hat{\bar f}(s,Y_s^m,Z^m_s,  \hat{G}(s,U^m_s))\right|+K \int_{\R_0} \left|\mathcal{H}(U_s^n(x))  - \mathcal{H}(U_s^m(x)) \right|  \nu(dx) \non\\
\leq &\left|\hat{\bar f}(s,Y_s^n,Z^n_s,  \hat{G}(s,U^n_s))-\hat{\bar f}(s,Y_s^m,Z^m_s,  \hat{G}(s,U^m_s))\right|+K \int_{\R_0} \left|\mathcal{H}'(\vartheta')\right|\left|U_s^n(x)  - U_s^m(x) \right|  \nu(dx),
  \end{align}
with 
$\min\left\{U_s^n(x), U_s^m(x) \right\}\leq\vartheta'\leq \max\left\{U_s^n(x), U_s^m(x) \right\}.$ 
   Using \ref{H4} and the bound for $U^n$ and $U^m$ again, we estimate  the integral from \eqref{initial2} by 
  \begin{align*}\label{initial3}
  &K \int_{\R_0} \left|\mathcal{H}(U_s^n(x))  - \mathcal{H}(U_s^m(x)) \right|  \nu(dx) \non\\
\leq
& K \int_{\R_0} \left|c_{R'}\max\left\{|U_s^n(x)|, |U_s^m(x)| \right\}\right|\left|U_s^n(x)  - U_s^m(x) \right|  \nu(dx),\non\\
\leq 
&K \int_{\R_0} c_{R'}\left ( A_{D\xi}(x) e^{\int_s^T a(v)dv}+  \int_s^T p(r,x) e^{\int_s^r a(v)dv} dr\!\right ) \left|U_s^n(x)  - U_s^m(x) \right|  \nu(dx).\non
  \end{align*} 
 Finally, by the Cauchy-Schwarz inequality, we arrive at
 \begin{align*}
  &\les \left|\f^n(s,Y^n_s, Z^n_s, U^n_s)-\f^n(s,Y^m_s, Z^m_s, U^m_s)\right|\non \\
\leq &\left|\hat{\bar f}(s,Y_s^n,Z^n_s,  \hat{G}(s,U^n_s))-\hat{\bar f}(s,Y_s^m,Z^m_s,  \hat{G}(s,U^m_s))\right|
\\&+K    C_2 \left (   \| A_{D\xi}\|^2  +  \int_0^T \|p(r,\cdot )\|^2 dr\right)^\frac{1}{2}   \left\|U_s^n - U_s^m \right\|,\non
  \end{align*} 
  for  $C_2= C_2\Big(c_{R'}, \int_0^T a(v)dv \Big) $.
Since $\hat{\bar{f}}$ and $\hat{G}$ satisfy a Lipschitz condition, the above inequality shows that also all 
$\f^n$ applied to $(Y^n,Z^n,U^n)$ and $(Y^m,Z^m,U^m)$ 
behave like Lipschitz functions with Lipschitz coefficients that do not depend on $n$ or $m$.\\
 
Exploiting this property, very similar methods as the standard procedure used in \cite[Proposition 2.2]{bbp} show that
there exists a constant $C>0$ (only dependent on the Lipschitz coefficients of $\hat{\bar{f}}$) such that 

\equal\label{primera}
&&\les \les  \|Y^n -Y^m \|^2_{\mathcal{S}_2}   + \|Z^n -Z^m\|^2_{L_2(W) } +  \|U^n -U^m\|^2_{L_2(\tilde N) }\non \\
&\le&C \E \int_0^T \left|\varphi\left(\bar f(s,Y_s^n,Z^n_s,  G(s,U^n_s)), \int_{\R_0}\!  \mathcal{H}^n(U_s^n(x),x)  \nu(dx)\right)\right.\non \\
&&\left.\quad\quad\quad-  \varphi\left(\bar f(s,Y_s^n,Z^n_s,  G(s,U^n_s)),\int_{\R_0}\!  \mathcal{H}^m(U_s^n(x),x)  \nu(dx)\right) \right|^2 ds.
\tionl
The mean value theorem applied to the second variable of $\varphi$ helps to estimate the latter term by
\equal\label{primeraymedia}
&&   C \E \int_0^T \left|\partial_w\varphi\left(\bar f(s,Y_s^{n},Z^{ n}_s,  G(s,U^{n}_s)),\vartheta\right)\right|^2 \non \\
&& \quad\quad \times \bigg ( \int_{\R_0} \left|\mathcal{H}(U_s^n(x)) \min\{1,n |x| \}  - \mathcal{H}(U_s^n(x)) \min\{1,m |x| \} \right|  \nu(dx)  \bigg  )^2 ds\non \\
& = & C \E \int_0^T \left|\partial_w\varphi\left(\bar f(s,Y_s^{ n},Z^{ n}_s,  G(s,U^{n}_s)),\vartheta\right)\right|^2 \\
&&  \times \bigg ( \int_{\R_0} |\mathcal{H}(U_s^n(x))| |\min\{1,n |x| \}  -  \min\{1,m |x| \} | \, \nu(dx)  \bigg  )^2 ds. \non 
\tionl 

 As in \eqref{initial}-\eqref{initial2} above, we continue to estimate inequalities \eqref{primera} and \eqref{primeraymedia} similarly by
\equal\label{segunda}
&& \les \les \|Y^n -Y^m \|^2_{\mathcal{S}_2}   + \|Z^n -Z^m\|^2_{L_2(W) } +  \|U^n -U^m\|^2_{L_2(\tilde N) } \non \\
&\le& C  K^2 \int_0^T\!\!\!\! \bigg ( \int_{\R_0} \bigg |  c_{R'}   \left ( A_{D\xi}(x) e^{\int_s^T a(v)dv}+  \int_s^T p(r,x) e^{\int_s^r a(v)dv} dr\!\right )^2\bigg | \non \\ 
&& \hspace{10em} \times    |\min\{1,n |x| \}  -  \min\{1,m |x| \}|    \nu(dx) \bigg )^2 ds   \\
& \le&  C K^2  T C_1  \left (   \| A_{D\xi}\|^2  +  \int_0^T \|p(r,\cdot )\|^2 dr\right)^2  <\infty,
 \tionl 
where  $C_1= C_1\Big(c_{R'}, \int_0^T a(v)dv \Big) $. The last estimate allows us to use dominated convergence. Hence, 
$$\|Y^n -Y^m \|^2_{\mathcal{S}_2}   + \|Z^n -Z^m\|^2_{L_2(W) } +  \|U^n -U^m\|^2_{L_2(\tilde N) } \to 0, \quad \text{ as } \quad m,n\to\infty,$$
because of $\lim_{n,m\to\infty} |\min\{1,n |x| \}  -  \min\{1,m |x| \}| =0 \text{ for all } x.$
This proves the existence of the limits
$$Y^n\xrightarrow[\mathcal{S}_2]{} Y, \quad Z^n\xrightarrow[L_2(W)]{} Z,\quad U^n\xrightarrow[L_2(\tilde N)]{} U.$$
Note that the triplet $(Y,Z,U)$ obeys the same bounds as all $(Y^n,Z^n,U^n)$ do.\bigskip 

{\bf Step 2}\\
It remains to show that $(Y,Z,U)$ indeed solves the BSDE given by $(\xi,\f)$:

By the convergence of $(Y^n,Z^n,U^n)$ to $(Y,Z,U)$, we know that for all $t\in {[0,T]}$
\begin{align*}
&Y_t^n\xrightarrow[L_2]{}Y_t,\quad \int_t^T Z_s^ndW_s\xrightarrow[L_2]{}\int_t^T Z_sdW_s,\\
&\int_{{]t,T]}\times\R_0}U^n_s(x)\tilde{N}(ds,dx)\xrightarrow[L_2]{}\int_{{]t,T]}\times\R_0}U_s(x)\tilde{N}(ds,dx),
\end{align*}
so three terms of the BSDE of $(\xi, \f^n)$ already converge to the respective terms of the one given by $(\xi, \f)$.

The last term  which needs
to converge to the right limit is $\int_t^T \f(s,Y^n,Z^n,U^n)ds$. Therefore consider
\begin{align*}
& \hspace{-3em}\left|\int_t^T \f^n(s,Y^n_s,Z^n_s,U^n_{s,\cdot})ds-\int_t^T \f(s,Y_s,Z_s,U_{s,\cdot})ds\right|  \\
\leq&  \int_0^T\left|\bar{f}(s,Y^n_s,Z^n_s,G(s,U^n_s))-\bar{f}(s,Y_s,Z_s,G(s,U_s)\right|ds\\
& + K\int_0^T\int_{\R_0}\left|\mathcal{H}^n(U^n_s(x),x)-\mathcal{H}(U_s(x))\right|\nu(dx)ds,
\end{align*}
where the constant $K$ is chosen in the same way as in the previous step, replacing $\mathcal{H}^m, U^m$ by $\mathcal{H}, U$. 
Having $(Y^n,Z^n,U^n)$ and $(Y,Z,U)$ estimated by the same bounds for all $n\in\mathbb{N}$, like in Step 4 of the proof of Theorem \ref{bounded-sol-x}, $\bar {f}$ acts as Lipschitz function and yields a constant $C$ with
\begin{align*}
& \int_0^T\left|\bar{f}(s,Y^n_s,Z^n_s,G(s,U^n_s))-\bar{f}(s,Y_s,Z_s,G(s,U_s)\right|ds\\
&\leq C\int_0^T\left(|Y^n_s-Y_s|+|Z^n_s-Z_s|+\|U^n_s-U_s\| \right)ds\xrightarrow[L_2]{n\to\infty}0.
\end{align*}
The only term left is now
\begin{align*}
&\int_0^T\int_{\R_0}\left|\mathcal{H}^n(U^n_s(x),x)-\mathcal{H}(U_s(x))\right|\nu(dx)ds=\\
&\int_0^T\int_{\R_0}\left|\mathcal{H}(U^n_s(x))\min\{1,n |x| \}-\mathcal{H}(U_s(x))\right|\nu(dx)ds,
\end{align*}
which approaches zero by  a similar dominated convergence argument as in the inequalities \eqref{primera} and \eqref{segunda} replacing $\min\{1,m |x|\}$ by $1$. 
Thus $(Y,Z,U)$ solves the BSDE $(\xi,\f)$.

{\bf Step 3}\\
This final step shows uniqueness of solutions to this equation in the  class  $ \mathcal{S}_\infty \times  L_\infty(W)\times L_{2,b}(\tilde N) .$ 
Let $(Y^j,Z^j,U^j), j=1,2$ be two solution to the BSDE $(\xi,\f)$ with bounds $(R^j,Q^j, 2R^j)$ as in \eqref{Y-estimate}, \eqref{Z-estimate} and 
\eqref{U-estimate}, and assume that  $A^j\in L_2(\nu)$  such that $|U^j_s(x)| \le A^j(x)$ 
for the respective solution processes.

We start, similarly to the last step (and to Step 4 from the proof of Theorem \ref{bounded-sol-x}), to consider $\bar{f}$ as a Lipschitz function. We look 
at the difference
\begin{align*}
\left|\varphi\left(\bar{f}(s,Y^1_s,Z^1_s,U^1_s)\!,\!\int_{\R_0}\!\mathcal{H}(U^1_s(x))\nu(dx)\right)\!-
\!\varphi\left(\bar{f}(s,Y^2_s,Z^2_s,U^2_s),\!\int_{\R_0}\!\mathcal{H}(U^2_s(x))\nu(dx)\right)\right|
\end{align*}
and estimate it by
\begin{align*}
\left|\bar{f}(s,Y^1_s,Z^1_s,U^1_s)-\bar{f}(s,Y^2_s,Z^2_s,U^2_s)\right|
+K\int_{\R_0}\left|\mathcal{H}(U^1_s(x))-\mathcal{H}(U^2_s(x))\right|\nu(dx). 
\end{align*}
The constant $K$ is chosen similarly as in Step 1, here using the bounds for $U^1$ and $U^2$.
Assumption \ref{H4} and the mean value theorem now imply that the last term is smaller than
\begin{align*}
&\left|\bar{f}(s,Y^1_s,Z^1_s,U^1_s)-\bar{f}(s,Y^2_s,Z^2_s,U^2_s)\right|
\\&
+K\int_{\R_0}c_{2R^1\vee 2R^2}\left(|U^1_s(x)|+|U^2_s(x)|\right)\left|U^1_s(x)-U^2_s(x)\right|\nu(dx).
\end{align*}
By the bounds $A^1, A^2$, and the Cauchy-Schwarz inequality, we arrive at the inequality
\begin{align*}
&\left|\bar{f}(s,Y^1_s,Z^1_s,U^1_s)-\bar{f}(s,Y^2_s,Z^2_s,U^2_s)\right|\\
&+K\int_{\R_0}c_{2R^1\vee 2R^2}\left(A^1(x)+A^2(x)\right)\left|U^1_s(x)-U^2_s(x)\right|\nu(dx)  \\
&\leq \left|\bar{f}(s,Y^1_s,Z^1_s,U^1_s)-\bar{f}(s,Y^2_s,Z^2_s,U^2_s)\right|+Kc_{2R^1\vee 2R^2}  \|A^1+A^2\|   \left\|U^1_s-U^2_s\right\| ,
\end{align*}
which shows that the standard procedure for Lipschitz generators (e.g. the one from \cite[Proposition 4.2]{SteinickeII}) is applicable. 
The uniqueness of the solution then follows.
\end{proof}

\begin{rem}
The setting of Theorem \ref{Hbsde} contains the example $$\mathcal{H}_\al(u) := \frac{e^{\al u} -\al u -1}{\al}$$ for some  fixed $\al>0.$ This type 
of generators  appears in BSDEs related to utility optimisation, see the  work of Morlais \cite{Morlais} and Becherer et al. \cite{BechererI}. 
\end{rem}
{\ch 
\section{ Locally Lipschitz BSDEs and PDIEs}\label{sec5}

In this section we apply our results to the correspondence between L\'evy-driven forward-backward SDEs with locally Lipschitz generators and partial 
differential-integral equations (PDIEs). In the case of Lipschitz generators this has been investigated before e.g. in \cite{bbp} and \cite{NualartSchoutens}, 
and in the Brownian setting for locally Lipschitz generators  in
\cite{CheriditoNam}. We recall the setting of \cite{bbp} (in a one-dimensional version) but 
relax the Lipschitz condition from there to a local Lipschitz condition:
\bigskip

Assume a generator function of the type $$\f(\omega,t,y,z,u)= f( \Psi_t(\omega),t,y,z,u),$$ with 
\begin{enumerate}[label={(\roman*)}]
 \item $f$ is Lipschitz in $y$
  and locally Lipschitz in $(z,u),$ 
 \item $f$ is non-decreasing in $u,$
 \item $f$ has a continuous partial derivative in the first variable bounded by $K>0$, which is also locally Lipschitz in $(y,z,u)$.
 \end{enumerate} 

Here, $(\Psi^t(v))_{(t,v)\in{[0,T]}\times\R}$ denotes a family of forward processes given by the SDEs
$$d\Psi^t_s(v)=b(\Psi^t_s(v))ds+\sigma(\Psi^t_s(v))dW_s+\beta(\Psi^t_{s-}(v),x)\tilde{N}(ds,dx),\quad s\in [t,T],$$
with $\Psi^t_t =v$ and the following requirements:

\begin{enumerate}
\item[(iv)] The functions $b\colon\R\to\R$ and $\sigma\colon\R\to\R$ are continuously differentiable with bounded derivative.
\item[(v)] $\beta\colon\R\times\R_0\to\R$ is measurable, satisfies 
\begin{align*}
\left|\beta(\psi,x)\right|\leq C_\beta(1\wedge|x|),\quad (\psi,x)\in\R\times\R_0
\end{align*}
and is continuously differentiable in $\psi$ with bounded derivative for fixed $x\in\R_0.$
\end{enumerate} 

Define the partial differential-integral operator
\equal\label{PDIE}
-\partial_tu(t,v)-\mathfrak{L}u(t,v)-\tilde{f}(v,t,u(t,v),\partial_vu(t,v),\mathfrak{B}u(t,v))&=&0,\quad(t,v)\in {[0,T]}\times\mathbb{R},\nonumber\\
 u(T,v)&=&\mathfrak{g}(v),\quad v\in\R,
\tionl
for a bounded, Lipschitz function $\mathfrak{g}$ and the operator $\mathfrak{L}=\mathfrak{A}+\mathfrak{K}$ given by 
\equa
&& \mathfrak{A}\varphi(v)=\frac{\sigma^2(v)}{2}\partial_v^2\varphi(v)+b(v)\partial_v\varphi(v),\\
&& \mathfrak{K}\varphi(v)=\int_{\R_0}\left(\varphi(v+\beta(v,x))-\varphi(v)-\beta(v,x) \partial_v\varphi(v)\right)\nu(dx),
\tion
and the integral operator $\mathfrak{B}$ given by
\equa
\mathfrak{B}\varphi(v)=\int_{\R_0}\left(\varphi(v+\beta(v,x))-\varphi((v)\right)\kappa(x)\nu(dx).
\tion

\bigskip

Denote by  $(Y^t(v),Z^t(v), U^t(v,x))$ the solution to the BSDE

\equal\label{BSDEPDIE}
&&Y^t_s(v)=\mathfrak{g}(\Psi^t_T(v))+\int_s^T f\left(\Psi_r^t(v),Y^t_r(v),Z_r^t(v),\int_{\R_0} U^t_r(v,x)\kappa(x)\nu(dx)\right)dr\nonumber\\
&&\quad\quad\quad\quad-\int_s^TZ_r^t(v)dr-\int_{]r,T]\times\R_0}U^t_r(v,x)\tilde{N}(ds,dx),\quad s\in [t,T],
\tionl
for $(t,v)\in [0,T]\times\R$.

For the notion of  a viscosity solution we refer to  \cite[Section 3]{bbp}. The following theorem is a  one-dimensional version of 
\cite[Theorem 3.4]{bbp}, where the generator may be locally Lipschitz in $(z,u).$
\begin{thm}\label{PTIEThm}
With the assumptions from this section, and the additional  conditions  \ref{f-prime}, \ref{growth}, \ref{Df-condition} and 
 \equal \label{boundDPsi} 
 \int_0^T\!\!\left ( \int_{\R_0} \left\|(r,x)\mapsto D_{r,x}\Psi_s\right\|^2_{L^\infty(\lambda \otimes \mathbb{P})}\nu(dx) \right )^{1/2} \!\!\!ds<\infty \,\,  and \,\,
 \left\|(r,x)\mapsto D_{r,x}\Psi_T\right\|^2_{L^\infty(\lambda \otimes \mathbb{P})} \in L_2(\nu), \notag \\
\tionl
the function 
$$u(t,v)=Y_t^t(v), \quad (t,v)\in [0,T]\times\R,$$ with $Y_t^t(v)$ given by \eqref{BSDEPDIE} is a viscosity solution  to the PDIE \eqref{PDIE}.
\end{thm}
\begin{proof}
By the consequences of Theorem \ref{bounded-sol-x}, the solution  $(Y^t(v),Z^t(v), U^t(v,x))$ may be regarded as solution to a Lipschitz BSDE 
provided that $f$ satisfies \ref{Axi}-\ref{bounded-derivatives}. We show that this is indeed the case.

The boundedness condition \eqref{boundDPsi} on $D\Psi$ and the Lipschitz continuity and boundedness property of $\mathfrak{g}$ imply that \ref{Axi} is satisfied.
Moreover, by the boundedness and differentiability assumption (in the first variable) on $f$, \ref{A-Df} is satisfied. The  conditions 
\ref{f-prime}, \ref{growth}, \ref{Df-condition} were assumed and since $f$ is Lipschitz in $y$ and locally Lipschitz in $(z,u)$, also \ref{Lipschitz} is satisfied. 

In the present generator, the function $g$ is given by $g(t,u)=u$, thus \ref{g-condition} is readily checked and
\ref{bounded-derivatives} follows as $f$ increases in $u$. 

Having verified \ref{Axi}-\ref{bounded-derivatives}, we may now assume that $Y^t(v)$ is a bounded solution to a BSDE with a Lipschitz generator fitting the assumptions of \cite[Theorem 3.4]{bbp}, which in turn guarantees that $(t,v)\mapsto Y^t_t(v)$ solves \eqref{PDIE}.
\end{proof}

\begin{rem} 
The conditions on  
$f$ in this section were taken from  \cite{bbp} and may surely be relaxed for Theorem \ref{PTIEThm} to hold. 
The proof of \cite[Theorem 3.4]{bbp} relies on the comparison theorem for BSDEs, which is valid in a more general setting than required in \cite{bbp}, see e.g. Theorem \ref{comparison}. Note also that the function $\kappa(x) = 1\wedge |x| $ can be generalised like in  \cite{bbp} to a function depending also on the forward process,   $\kappa(\Psi_t(\omega), x), $ under suitable conditions.
\end{rem}

We finish with an example where \eqref{boundDPsi} is satisfied   for  $D_{r,x}\Psi$ with $x=0.$ For  $x\neq 0$ it follows from  Lemma \ref{functionallem} and  Lemma \ref{path-shift}   that  $D_{r,x}\Psi$ is bounded if  $\Psi$  is. 
\newpage 
\begin{example}  We have for $t \le r\le s\le T$ that
 \equa
D_{r,0}\Psi_s &= & \sigma(\Psi_r) + \int_r^s b'(\Psi_u) D_{r,0}\Psi_u  du+   \int_r^s \sigma'(\Psi_u) D_{r,0}\Psi_u   dW_u \\
&&+   \int_{]r,s]\times\R_0} \partial_\psi\beta(\Psi^t_{u-},x)D_{r,0}\Psi_u    \tilde{N}(du,dx) 
\tion
which implies that 
 \equal\label{MderivBd}
D_{r,0}\Psi_s &=&  \sigma(\Psi_r) \exp{\left (\int_r^s b'(\Psi_u)   du+   \int_r^s \sigma'(\Psi_u)   dW_u  - \frac{1}{2} \int_r^s( \sigma'(\Psi_u) )^2  du
   \right )} \notag\\
&& \times \prod_{r < u \le s} \big(1+ \partial_\psi\beta(\Psi_{u-}, \Delta L_u )\big ) \exp{\left (- \int_r^s \int_{\R_0} \partial_\psi\beta(\Psi_{u},x)  \nu (dx) du  \right )}. 
\tionl
We first check conditions to have the stochastic integral $  \int_r^s \sigma'(\Psi_u)   dW_u$ bounded from above. By It\^o's formula,
 \equa
 \log (|\sigma(\Psi_s)|) &=&  \log (|\sigma(\Psi_r)|) +  \int_r^s \sigma'(\Psi_u)   dW_u + \int_r^s \frac{\sigma'(\Psi_u) }{ \sigma(\Psi_u) } b(\Psi_u)    du \\
 && + \frac{1}{2} \int_r^s\sigma(\Psi_u) \sigma''(\Psi_u) - ( \sigma'(\Psi_u) )^2    du \\
 &&+ \int_{]r,s]\times\R_0} [ \log (|\sigma(\Psi_{u-})  +\beta(\Psi_{u-},x)   |)  - \log (|\sigma(\Psi_{u-}) |)]  \tilde{N}(du,dx)  \\
 && +  \int_r^s \int_{\R_0}\Big [ \log (|\sigma(\Psi_{u-})  +\beta(\Psi_{u-},x)   |)  - \log (|\sigma(\Psi_{u-}) |)  \\ &&
 \hspace{15em} -    \frac{\sigma'(\Psi_u) }{ \sigma(\Psi_u) }  \beta(\Psi_{u-},x) \Big ] \nu (dx) du
 \tion
so that 
 \equa
  \int_r^s \sigma'(\Psi_u)   dW_u &=&  \log \left (\frac{|\sigma(\Psi_s)|}{ |\sigma(\Psi_r)|  } \right) \\
  && +  \int_r^s \left [ \frac{( \sigma'(\Psi_u) )^2  }{2}   -  \frac{\sigma'(\Psi_u)  b(\Psi_u)  }{ \sigma(\Psi_u) }   -   \frac{1}{2}   \sigma(\Psi_u) \sigma''(\Psi_u)  \right ] du \\
  &&+\int_r^s \int_{\R_0}    \frac{\sigma'(\Psi_u) }{ \sigma(\Psi_u) }  \beta(\Psi_{u},x) \nu (dx) du \\
  && - \log \Big ( \prod_{r < u \le s} \Big |1 + \frac{\beta(\Psi^t_{u-}, \Delta L_u )}{ \sigma(\Psi_{u-}) } \Big |\Big ), 
 \tion
which is bounded from above if 
\begin{itemize}
\item[*] $ c^{-1} \le  \sigma(\Psi_u)    \le c$ for all $u \in [t,T],$ for some $c>0,$ \\
\item[*] $ \frac{( \sigma'(\Psi_u) )^2  }{2}   -  \frac{\sigma'(\Psi_u)  b(\Psi_u)  }{ \sigma(\Psi_u) }   -   \frac{1}{2}   \sigma(\Psi_u) \sigma''(\Psi_u) + \int_{\R_0}    \frac{\sigma'(\Psi_u) }{ \sigma(\Psi_u) }  \beta(\Psi_{u},x) \nu (dx) \le c'$ for all $u \in [t,T],$ for some $c'>0,$ \\
\item[*] $\beta \ge 0.$
\end{itemize}
 
Additional conditions that guarantee the boundedness of the right hand side of equation \eqref{MderivBd} are 
\begin{itemize}
\item[*] $ b'(\Psi_u) -\frac{(\sigma'(\Psi_u))^2}{2}  \le \tilde{c}$ for all $u \in [t,T],$ for some $\tilde c>0,$ \\
\item[*] $ -1<\partial_\psi \beta(\Psi_u,x)\le 0$ for all $u \in [t,T]$,\\
\item[*] $-\int_\mathbb{R}\partial_\psi\beta(\Psi_u,x)\nu(dx) \le \tilde{c}'$,  for all $u \in [t,T]$, for some $\tilde{c}'>0$.
\end{itemize}
\end{example}
}

 \section*{Acknowledgement}
Alexander Steinicke  is supported by the Austrian Science Fund (FWF): Project F5508-N26, which is
part of the Special Research Program ``Quasi-Monte Carlo Methods: Theory and Applications''.

Christel Geiss would like to thank the  Erwin Schr\"odinger Institute, Vienna, for hospitality and support, where  a part of this work was written.

\appendix \section{Appendix}
\subsection*{Malliavin differentiability for Lipschitz generators}\label{app}

For the terminal value $\xi$ and the function  $\f$ with  
\begin{align*}
\f(\om,t,y,z, \bu)=  f \left(X(\om), t,y, z,      \int_{\R_0}  g(s, \bu (x)) \kappa (x) \nu(dx) \right)
\end{align*} 
we agree upon the following assumptions:  \bigskip

($\A_{\xi}$)  $\xi \in \mathbb{D}_{1,2}.$    \bigskip \\   
(${\A}_f$)  \vspace*{-1.1em}  
\begin{enumerate}
[font=\itshape,labelindent=7mm,leftmargin=*,topsep=0mm,label= { \alph*)} ]
\item  \label{f-meas}
 $f\colon D{[0,T]}\times{[0,T]}\times\R^3 \to \R$ is jointly measurable, adapted to $(\mathcal{G}_t)_{t\in{[0,T]}}$ defined in  \eqref{filtrationG-t},
 and for all $t\in [0,T]$  and  $i=1,2,3,$ \,$\exists  \,\, {\partial_{\eta_i}}f(\tx,t,\eta),$ and the functions
\equa
 \R^3 \ni \eta   \mapsto {\partial_{\eta_i}} f(\tx,t,\eta)
\tion
are  
continuous.
 \item   \label{f(000)-Ltwo}  There exist functions $k_f\in L_1([0,T])$, $K_f\in L_2(W) $, 
  such that 
  $$ |f(X,t,0,0,0)|\leq k_f(t)+K_f(t),\quad\mathbb{P}\text{-a.s.}$$
\item   \label{f-Lip}

 $f$ satisfies the following Lipschitz condition: There exist nonnegative functions \\$a\in L_1([0,T]), {b\in L_2([0,T])}$ such that for all 
 $t\in{[0,T]}, \newline(y,z,u), (\tilde y, \tilde{z},\tilde{u})\in\R^3$
 $$ |f\left(\tx,t, y,z,u \right)- f\left(\tx,t,\tilde{y},\tilde{z} ,\tilde u\right)|\leq a(t) |y- \tilde y|+b(t)(|z-\tilde{z}|+|u-\tilde{u}|).$$  
 \item   \label{f-Mall-diff-bounded}  Assume  there is a nonnegative random field $\Gamma\in\mathrm{L}_2(\lambda\otimes\PP \otimes \m)$, and a nonnegative ${\rho^D\in L_2(\m)}$ such that
 for all random vectors $G=(G_1,G_2,G_3) \in (\mathrm{L}_2)^3$   and      for a.e. $t$ it holds
\equa
 \left|\left(D_{s,x} f\right)(t,G)\right|\le \Gamma_{s,x}(t)+\rho^D_{s,x}|G|,\quad   \mathbb{P}\otimes\m \text{-a.e}.
\tion
where $\left(D_{s,x}f\right)(t,G) := D_{s,x}f(X,t,\eta)\mid_{\eta=G}.$
 \item \label{f-Mall-diff-Lip}
 $f(X,t,\eta)\in\mathbb{D}_{1,2}$  for all $(t,\eta)\in{[0,T]}\times\R^3,$ and  $\forall t\in{[0,T]}$, $\forall N\in\N $ $\
\exists\ K^t_N\in\bigcup_{p> 1}\mathrm{L}_p$ such that  for a.a. $\omega$ 
\begin{equation*}
\begin{split}
&  \forall \eta, \tilde \eta\in B_N(0): \\
&\|\left(D_{.,0}f(X,t,  \eta)\right)(\omega)-\left(D_{.,0}f(X,t,  \tilde \eta)\right)(\omega)\|_{ L_2({[0,T]})}
   <K^t_N(\omega)\left|\eta-  \tilde \eta\right|,
\end{split}
\end{equation*}
where for $D_{.,0}f(X,t,  \eta)$ we always take a progressively measurable version in $t$.

\item  \label{g-condition}  $g:[0,T]\times \R \to \R$ is jointly measurable,
 $g(t,\cdot) \in \mathcal{C}^1(\R)$ with $g(t,0)=0$, bounded derivative $|g'(t,\cdot)|\leq L_g,$ and  $\kappa \in\mathrm{L}_2(\nu).$
\end{enumerate}

For similar results on differentiability of BSDEs with jumps in the L\'evy case, see \cite{ImkellerDelong}, \cite{Delong} or \cite{GeissSteinII}.
The following result  generalises \cite[Theorem 4.4]{GeissSteinII} and -- up to the time delay --  \cite[Theorem 4.1]{ImkellerDelong}.\bigskip

\begin{thm}\label{diffthm}
Assume (${\A}_\xi$) and (${\A}_f$).  Then the following assertions hold.
\begin{enumerate}[label={(\roman*)}]
\item \label{alef}  For $\m$- a.e. $(r,v) \in [0,T]\times \R$  there exists a unique solution $(\mathcal{Y}^{r,v},\mathcal{Z}^{r,v}, 
\mathcal{U}^{r,v}) $ $\in  \mathcal{S}_2\times L_2(W) \times L_2(\tilde{N})$ to the BSDE
\equal 
\mathcal{Y}^{r,v}_t &=& D_{r,v} \xi + \int_t^T F_{r,v}\kla s, \mathcal{Y}^{r,v}_s, \mathcal{Z}^{r,v}_s,  \mathcal{U}^{r,v} _s\mer ds  \non\\ 
&&   -\int_t^T {\mathcal{Z}}^{r,v}_{s}dW_s-\int_{{]t,T]}\times\R_0}\mathcal{U}^{r,v}_{s,x}\tilde N(ds,dx), \quad 0\le r \le t \le T,   \label{cY-eqn} \non  \\
  \mathcal{Y}^{r,v}_s &=& \mathcal{Z}^{r,v}_s =   \mathcal{U}^{r,v}_s = 0,    \quad   0\le s< r\le T , \label{U-V-equation}
\tionl
where   $ \Theta_s :=(Y_s,Z_s, G(s,U_s))$ and
\equa
 \les \les\les F_{r,0} (s,  y,z,\bu)
 &\quad:=& (D_{r,0}  f) (s, \Theta_s) + \partial_y f(s, \Theta_s) y   + \partial_z f(s, \Theta_s) z \\
&& + \partial_u  f (s,\Theta_s) \int_{\R_0}  \partial_u   g(s,  U_s(v)) \bu(v)    \kappa (v) \nu(dv), 
\tion
and  for $v \neq 0,$
\equa
 F_{r,v} (s,   y,z,\bu)&:=& (D_{r,v}  f ) (X, s, \Theta_s )   \\
 && +   f  (X+ v\one_{[r,T]}, s, \Theta_s  +  (y,z,  G(s,U_s + \bu )))    - f (X+ v\one_{[r,T]},s, \Theta_s ). \\
\tion 
\item \label{bet} For the solution $(Y,Z,U)$ of \eqref{bsde} it holds
\equal \label{Y-and-Z-in-D12}
Y, Z   \in L_2([0,T];\DD), \quad U\in L_2([0,T]\times\R_0;\DD),
\tionl
and $D_{r,y}Y$ admits a c\`adl\`ag version for $\m$- a.e. $(r,y) \in  [0,T]\times\R.$
\item \label{gimmel} $(D Y,D Z, D U)$ is a version of $(\mathcal{Y},\mathcal{Z}, \mathcal{U}),$  i.e. for $\m$- a.e. $(r,v)$ it solves
\begin{align} \label{diffrep}
D_{r,v}Y_t = &D_{r,v}\xi+\int_t^T F_{r,v}\left(s,D_{r,v}Y_s, D_{r,v}Z_s,D_{r,v} U_s \right)ds \\
&-\int_t^T D_{r,v}Z_sdW_s-\int_{{]t,T]}\times{\R_0}} D_{r,v} U_{s}(x)\tilde{N}(ds,dx), \quad 0\le r \le t \le T. \non
\end{align}
\item \label{dalet} Setting  $D_{r,v}Y_r(\omega):= \lim_{t\searrow r}D_{r,v}Y_t(\omega) $
for all $(r,v,\omega)$ for which  $ D_{r,v}Y$ is c\`adl\`ag and \\${D_{r,v}Y_r(\omega):=0}$ otherwise, we have that 
{\ch \begin{align*} \left(\left(\E[ D_{r,0}Y_r | \ftn_{r-}] \right)_{r\in{[0,T]}}\right) &\text{ is a version of  }(Z_{r})_{r\in{[0,T]}},\\
\left(\left(\E[ D_{r,v}Y_r  | \ftn_{r-}]  \right)_{r\in{[0,T]}, v\in\R_0}\right) &\text{ is a version of  }(U_{r}(v))_{r\in{[0,T]}, v\in\R_0}.
\end{align*}}
\end{enumerate}
\end{thm}

\subsection*{Proof of  Theorem \ref{diffthm}} 

 Let us start with a lemma providing estimates for the Malliavin derivative of the generator.

\begin{lemma}\label{destimlem}
Let $G=(G_1,G_2,G_3) \in (L_2)^3$ and $\Phi=(\Phi_1,\Phi_2,\Phi_3)\in(L_2(\PP\otimes\m))^3$. If $f$ satisfies (${\A}_f$)
it holds for $\mathbb{P} \otimes \m$-a.a. $(\omega,r,v), v \neq 0,$ that
\equal \label{generator-estimate}
&& \hspace{-3em}|f(X+v\one_{[r,T]},t,G+\Phi_{r,v})
-f\left(X,t,G\right)| 
\non\\&&
\le  a(t) \left|\Phi_{1,r,v}\right|+b(t)(|\Phi_{2,r,v}|\!+\!|\Phi_{3,r,v}|)+\Gamma_{r,v}(t)+{\rho}_{r,v}^D |G|.
\tionl
Moreover, for  $G \in (\DD)^3$ 
it holds   $f(X,t,G) \in \DD$ and 
\equal \label{Dgenerator-estimate}
  |D_{r,v}f\left(X,t,G\right)| \!\! &\le& \! 
   a(t) \left|D_{r,v} G_1\right|+b(t)(\left|D_{r,v}G_2\right|+\left|D_{r,v}G_3\right|) +\Gamma_{r,v}(t)+{\rho}_{r,v}^D|G|, 
   \quad \mathbb{P}\otimes \m \text{-a.e.} \non \\
   &&
\tionl

\end{lemma}
\begin{proof}
According to Lemma \ref{path-shift} we may replace $X$ by  $X+v\one_{[r,T]}$ and get from the  Lip\-schitz property (${\A}_f$) $\ref{f-Lip}$
that     
\begin{align*}
 &\big | f(X+v\one_{[r,T]},t,G+\Phi_{r,v}) - f(X+v\one_{[r,T]},t,G)  \big | 
  \le    a(t) \left|\Phi_{1,r,v}\right|+b(t)\left(\left|\Phi_{2,r,v}\right|+\left|\Phi_{3,r,v}\right|\right)
\end{align*}
for  $\mathbb{P}\otimes \m $-a.e. $(\om, r,v)$ with $v\neq 0.$ 
From (${\A}_f$) $\ref{f-Mall-diff-bounded}$ one concludes then \eqref{generator-estimate}.

For $v \neq 0$ we apply Lemma \ref{functionallem} to get $$D_{r,v}f\left(X,t,G\right)=  f(X+v\one_{[r,T]},t,G+D_{r,v}G)     -  f\left(X,t,G\right),$$
and hence \eqref{Dgenerator-estimate} follows from \eqref{generator-estimate}.
 In the case  $v=0$, \cite[Theorem 3.12]{GeissSteinII} implies  that under the assumptions  (${\A}_f$)   $\ref{f-meas}$ and  (${\A}_f$) $\ref{f-Mall-diff-Lip}$
the Malliavin derivative $D_{r,0}f(X,t,G)$ exists and it holds that 
\equal \label{derivative}
D_{r,0}f(X,t,G) &=&   (D_{r,0}f)(t,G) + \partial_{\eta_1}f (X,t,G)D_{r,0} G_1
 +   \partial_{\eta_2}f(X,t,G)D_{r,0} G_2 \non\\
  &&  +\partial_{\eta_3} f(X,t,G)D_{r,0} G_3
\tionl
for $\mathbb{P}\otimes \lambda$-a.a. $(\omega,r)\in \Omega\times{[0,T]}.$ Relation \eqref{Dgenerator-estimate} follows from conditions (${\A}_f$) $\ref{f-Lip}$
 and  \ref{f-Mall-diff-bounded} using
that the partial derivative $\partial_{\eta_1}f(X,t,\eta)$ is bounded by $a(t)$ and the derivatives $\partial_{\eta_i}f(X,t,\eta)$, $i=2,3$ are bounded by $b(t)$.
\end{proof}

{\bf Proof of Theorem \ref{diffthm}.} 
The core of the proof is to conclude assertion \ref{bet} which is 
done by an iteration argument.
To simplify the notation, in most places we do not  mention the dependency of $f$ on $X.$  \bigskip

\ref{alef} For those $(r,v)$ such that  $D_{r,v}\xi \in L_2$  the existence and uniqueness of a solution 
$(\mathcal{Y}^{r,v},\mathcal{Z}^{r,v},\mathcal{U}^{r,v})$ to \eqref{U-V-equation} follows from Theorem  \ref{existence} 
since $F_{r,v}$ meets the assumptions of the theorem.\\
By the a priori estimate shown in \cite[Proposition 4.2]{SteinickeII}, a solution  to a BSDE satisfying \ref{1} - \ref{3} depends continuously on the terminal condition, i.e. the mapping
\[
L_2 \to  L_2(W)\times L_2(W) \times L_2(\tilde{N})  \colon \zeta \mapsto (\mathcal{Y}^\zeta,\mathcal{Z}^\zeta,\mathcal{U}^\zeta)
\]
is continuous.  The existence of a jointly measurable version of
 \[
 (\mathcal{Y}^{r,v},\mathcal{Z}^{r,v}, \mathcal{U}^{r,v}), \quad  (r,v) \in [0,T]\times \R \]
 follows then by approximating  $D\xi$ (which is measurable  w.r.t.~$(r,v).$) by simple functions in $L_2(\PP \otimes\m).$ 
Joint measurability (for example for $\mathcal{Z}$) in all arguments can be gained by identifying 
  the spaces
$$L_2(\lambda,L_2(\mathbb{P}\otimes\m))\cong L_2(\lambda \otimes\mathbb{P}\otimes\m).$$  The quadratic integrability with respect to $(r,v)$  also follows from 
\cite[Proposition 4.2]{SteinickeII} since $\xi \in \DD .$ \smallskip \\
\ref{bet} We use the iteration scheme introduced in \cite{PP1}. Starting in our setting with $(Y^0,Z^0,U^0)=(0,0,0)$, we get $Y^{n+1}$ by taking the optional projection which implies that  $\PP$-a.s.
\equal \label{Yn}
Y^{n+1}_t=\E _{t}\left(\xi+\int_t^T f\left(s,Y^n_{s}, Z^n_s, G(s,U_s^n)\right)ds\right).
\tionl
 The processes ${Z}^{n+1}, U^{n+1}$
one gets by the martingale representation theorem w.r.t. $dW_s+N(ds,dx)$ (see, for example,  \cite{Applebaum}):
\equal   \label{zndef}
&& \les\xi+\int_0^T  f\left(s,Y^n_{s}, Z^n_s, G(s,U_s^n)\right)ds \\
&=&\E \left(\xi+\int_0^T  f\left(s,Y^n_{s}, Z^n_s, G(s,U_s^n)\right)ds\right)+\int_0^T {Z}^{n+1}_{s}dW_s+\int_{{]0,T]}\times\R_0}{U}^{n+1}_{s,x} \tilde N(ds,dx).\non
\tionl

\textbf{Step 1.}\\
In this first step we will show convergence of the so defined sequence $(Y^n,Z^n,U^n)\to(Y,Z,U)$ in $L_2(W)\times L_2(W)\times L_2(\tilde N)$.

Equations \eqref{Yn} and \eqref{zndef} mean that
\begin{align*}
Y_t^{n+1}= & \,\xi+\int_t^Tf\left(s,Y^n_{s}, Z^n_s, G(s,U_s^n)\right)ds
-\int_t^T Z^{n+1}_sdW_s-\int_{{]t,T]}\times\R_0}{U}^{n+1}_{s,x} \tilde N(ds,dx),
\end{align*}
which can be considered as BSDE with a generator not depending on the $y,z$ and $u$ variables.

With $\Delta Y^{n+1}=Y^{n+2}-Y^{n+1}$ and $\Delta Y^n=Y^{n+1}-Y^n$ and similar notations for the $Z$ and $U$ processes, we get, with the help of It\^o's formula ($\gamma \in L_2([0,T])$ will be determined later)
\begin{align} \label{ito-applied}
&\les e^{\int_0^t \gamma(s)ds}\left|\Delta Y^{n+1}_t\right|^2
+\int_t^T e^{\int_0^s \gamma(\tau)d\tau}\left(\gamma(s)\left|\Delta Y^{n+1}_s\right|^2+|\Delta Z^{n+1}_s|^2+  \|U^{n+1}_s\|^2 \right)ds \non \\
=&\int_t^Te^{\int_0^s \gamma(\tau)d\tau}2\Delta Y^{n+1}_s\left(f(s,Y^{n+1},Z^{n+1}_s,G(s,U^{n+1}_s))-f(s,Y^n_s,Z^n_s,G(s,U^n_s))\right)ds\\
&-M(t).    \non
\end{align}
In this equation, $M(t)$ consists of the stochastic integrals
\begin{align*}
\int_t^Te^{\int_0^s \gamma(\tau)d\tau}2\Delta Y^{n+1}_s \Delta Z^{n+1}_s dW_s+\int_{]t,T]\times\R_0}\left((\Delta U^{n+1}_s+\Delta Y^{n+1}_s)^2-|\Delta U^{n+1}_s|^2\right) \tilde N(ds,dx).
\end{align*}
By a standard procedure (see \cite[Proposition 4.1]{SteinickeII} for the present setting), one  concludes  from $(Y^n,Z^n,U^n)\in L_2(W)\times L_2(W)\times L_2(\tilde N)$  
that $Y^{n+1} \in \mathcal{S}_2.$
This fact, together with the Burkholder-Davis-Gundy inequality implies  that $\E M(t)=0$. 

We use conditions $(\A_f)$ \ref{f-Lip} and \ref{g-condition} and apply Young's inequality to the resulting terms to get 
\begin{align*}
&2 |\Delta Y^{n+1}_s|\, |f(s,Y^{n+1},Z^{n+1}_s,G(s,U^{n+1}_s))-f(s,Y^n_s,Z^n_s,G(s,U^n_s))| \\
&\leq 2|\Delta Y^{n+1}_s|\left(a(s)|\Delta Y^n_s|+(1\vee L_g\|\kappa\|)b(s)(|\Delta Z^n_s|+\|\Delta U^n_s\|)\right) \\
&\leq\left(2\left(a(s)+2(1\vee L_g\|\kappa\|)^2b(s)^2\right)|\Delta Y^{n+1}_s|^2+\frac{a(s)}{2}|\Delta Y^n_s|^2+\frac{|\Delta Z^n_s|^2+\|\Delta U^n_s\|^2}{2}\right).
\end{align*}
The right hand side increases if we replace  the factor $\tfrac{a(s)}{2}$ before $|\Delta Y^n_s|^2$  by $\tfrac{a(s)+1}{2}.$  Thus  \eqref{ito-applied},  after using this inequality and taking expectations turns into
\begin{align*}
&\E e^{\int_0^t \gamma(s)ds}\left|\Delta Y^{n+1}_t\right|^2+\E\int_t^T e^{\int_0^s \gamma(\tau)d\tau}\left(\gamma(s)\left|\Delta Y^{n+1}_s\right|^2+|\Delta Z^{n+1}_s|^2+\|U^{n+1}_s\|^2\right)ds\\
&\leq \E\int_t^Te^{\int_0^s \gamma(\tau)d\tau}  \bigg ( 2( a(s)+(1\vee L_g\|\kappa\| )^2b(s)^2  )|\Delta Y^{n+1}_s|^2\\
&\quad\quad\quad+\frac{a(s)+1}{2}|\Delta Y^n_s|^2+\frac{|\Delta Z^n_s|^2+\|\Delta U^n_s\|^2}{2}\bigg)ds.
\end{align*}
Setting $\gamma=1+3a+2(1\vee L_g\|\kappa\|)^2b^2$   and omitting the first term of the inequality,  we have for $t=0$ that 
\begin{align*}
&\E\int_0^T e^{\int_0^s \gamma(\tau)d\tau}\left((1+a(s))\left|\Delta Y^{n+1}_s\right|^2+|\Delta Z^{n+1}_s|^2+\|U^{n+1}_s\|^2\right)ds\non\\
&\leq \frac{1}{2}\E\int_0^Te^{\int_0^s \gamma(\tau)d\tau}\left((1+a(s))|\Delta Y^n_s|^2+|\Delta Z^n_s|^2+\|\Delta U^n_s\|^2\right)ds.
\end{align*} 
The last inequality states that the sequence $(Y^n,Z^n,U^n)_{n\geq 0}$ is subject to a contraction in the Banach space of all $(\bar{y},\bar{z},\bar{u})\in L_2(W)\times L_2(W)\times L_2(\tilde N)$, such that
\begin{align*} 
\|(\bar{y},\bar{z},\bar{u})\|^2_{1+a,\gamma}:=&\left\|e^{\int_0^\cdot \gamma(\tau)d\tau}(1+a)\bar{y}\right\|^2_{L_2(W)}
+\left\|e^{\int_0^\cdot \gamma(\tau)d\tau}\bar{z}\right\|^2_{L_2(W)} +\left\|e^{\int_0^\cdot \gamma(\tau)d\tau}\bar{u}\right\|^2_{L_2(\tilde{N})}<\infty.
\end{align*} 
This norm is stronger than $\sqrt{\|\cdot\|^2_{L_2(W)}+\|\cdot\|^2_{L_2(W)}+\|\cdot\|^2_{L_2(\tilde{N})}}$ on this space, hence the Picard iteration converges to the unique fixed point $(Y, Z, U)$. 

\textbf{Step 2.}\\
Our aim in this step 
is to show that $Y^n, Z^n$ and $U^n$  are  uniformly bounded in $n$ as elements of   $L_2(\lambda ; \DD)$ and $L_2(\lambda\otimes \nu  ; \DD),$ respectively. 
This will follow  from \eqref{boundedYn} below.   \\
We recall the notation for $M$ and $\m$ from \eqref{measureM} and define for $n\geq 0$, $$\underline{Z}^{n}_{t,x}=\begin{cases}Z^n_t, & x=0,\\ U^n_{t}(x), & x\neq 0.\end{cases}$$\\
 Given that $Y^n , Z^n \in L_2( \lambda; \DD)$ and $U^n \in L_2(\lambda\otimes \nu; \DD)$ one can infer that this also 
 holds for $n+1\colon$   Indeed, (${\A}_f$) $\ref{g-condition}$ implies that $G(s,U^n_s) \in \DD$ for a.e. $s$ and 
 \equal \label{g-integral-est}
   \left | D_{r,v} G(s,U^n_s) \right |  \le L_g \|\kappa\| \| D_{r,v} U^n_s   \|.
  \tionl
 From \cite[Theorem 3.12]{GeissSteinII} and Lemma \ref{destimlem} we conclude that $f(X,s,Y^n_s, Z^n_s,  G(s,U^n_s)) \in \DD.$  The above estimate and \eqref{Dgenerator-estimate} as well as the Malliavin differentiation rules shown by Delong and Imkeller in  \cite[ Lemma 3.1. and Lemma 3.2.]{ImkellerDelong}  imply that  $Y^{n+1}$ as defined in \eqref{Yn}   is   in   $L_2( \lambda; \DD).$  Then we conclude  that   both stochastic integrals  in \eqref{zndef}  are in  $ \DD$ and
 \cite[ Lemma 3.3.]{ImkellerDelong}  implies that for the corresponding integrals one has  $Z^{n+1} \in L_2( \lambda; \DD)$ and $U^{n+1} \in L_2(\lambda\otimes \nu; \DD).$
 Especially,  we get for $t\in{[0,T]}$ that  $\PP$ -a.e. 
\equal  \label{D-of-y-n+1}
D_{r,v}Y^{n+1}_t&=&D_{r,v}\xi+\int_t^T D_{r,v}f\left(X,s,Y^n_s, Z^n_s, G(s,U^n_s) \right)ds  \non\\
&&-\int_{{]t,T]}\times\R}D_{r,v}\underline{Z}^{n+1}_{s,x} M(ds,dx), \text{ for } \m \text{ - } a.a. \,(r,v) \in [0,t]\times \R , \non \\
D_{r,v}Y^{n+1}_t&=&0  \quad  \text{ for } \m  \text{ - } a.a. \,(r,v) \in (t,T]\times \R,  \non \\
D_{r,v} \underline{Z}^{n+1}_{t,x}&=&0  \quad \text{ for } \m\otimes\mu \text{ - } a.a. \,(r,v,x) \in (t,T]\times \R^2.
\tionl
Since by \cite[Theorem 4.2.12] {Applebaum}  the process 
$\big (\int_{{]0,t]}\times\R}D_{r,v}\underline{Z}^{n+1}_{s,x}M(ds,dx)\big)_{t\in [0,T]},$ admits a c\`adl\`ag version, we may take a c\`adl\`ag version of both sides. \\
By It{\^o}'s formula, we conclude that for $0< r<t$ and $\beta\in L_1([0,T])$ it holds
\equa
&e^{\int_0^T\beta(s)ds}(D_{r,v}\xi)^2=&e^{\int_0^t\beta(s)ds}(D_{r,v}Y^{n+1}_t)^2+\int_t^T\beta(s) e^{\int_0^s\beta(\tau)d\tau}(D_{r,v}Y^{n+1}_s)^2 ds \\
&&-2\int_t^T e^{\int_0^s\beta(\tau)d\tau}\big [D_{r,v}f\left(X,s,Y^n_s,Z^n_s, G(s,U^n_s) \right)\big]D_{r,v}Y^{n+1}_sds\\
&&+\int_{{]t,T]}\times\R} e^{\int_0^s \beta(\tau)d\tau}[2(D_{r,v}Y^{n+1}_{s-} )D_{r,v}\underline{Z}^{n+1}_{s,x} \\
&& \quad \quad \quad   \quad \quad \quad   \quad          +  \one_{\R_0}(x)  ({D_{r,v}\underline{Z}^{n+1}_{s,x}})^2 ]M(ds,dx)\\
&&+\int_{{]t,T]}\times\R}e^{\int_0^s\beta(\tau)d\tau}(D_{r,v}\underline{Z}^{n+1}_{s,x})^2ds\mu(dx), \quad \mathbb{P}\otimes\m\text{ - } a.e.
\tion
By \eqref{Dgenerator-estimate}, the requirements of  the a priori estimate \cite[Proposition 4.1]{SteinickeII} are met, which shows that $$\E\sup_{t\in {[0,T]}}\left|D_{r,v}Y^{n+1}_{t}\right|^2<\infty, \quad\mathbb{P\otimes\m}\text{-a.e.}$$ Thus, the integral w.r.t.~$M$ is a uniformly integrable martingale and hence has expectation zero. Therefore, using \eqref{D-of-y-n+1}, we have for $0< u<t \le T$ that
\equa
&& \les\E e^{\int_0^t\beta(s)ds}(D_{r,v}Y^{n+1}_t)^2+\E\int_{{]r,T]}\times\R}e^{\int_0^s\beta(\tau)d\tau}(D_{r,v}\underline{Z}^{n+1}_{s,x})^2ds\mu(dx)  \non\\
&\le&
e^{\int_0^T\beta(s)ds }\E(D_{r,v}\xi)^2 
+ 2\int_r^T e^{\int_0^s\beta(\tau)d\tau}\E\left| \big [ D_{r,v}f\left(X,s,Y^n_s,Z^n_s, G(s, U^n_s)\right)\big]D_{r,v}Y^{n+1}_s\right|ds  \non\\
&&-\E\int_r^T \beta(s)e^{\int_0^s\beta(\tau)d\tau}(D_{r,v}Y^{n+1}_s)^2 ds.
\tion
Similar as in Step  1  we estimate the integrand containing the generator. Here we use  Lemma  \ref{destimlem}  and \eqref{g-integral-est}, and then again  Young's inequality:
\equa
 && \les 2\left| \big [ D_{r,v}f\left(X,s,Y^n_s,Z^n_s, G(s,U^n_s)\right) \big ]D_{r,v}Y^{n+1}_s\right| \\
  &  \le 
  & 2\left| D_{r,v}Y^{n+1}_s\right|   \Big( \left|\Gamma_{r,v}(s)\right|  +a(s)\left|D_{r,v}Y^{n}_s\right|  +b(s)\left(\left|D_{r,v}Z^n_s\right|+L_g\|\kappa\| \| D_{r,v} U^n_s   \| \right)\\
&&  \quad\quad \quad +\rho_{r,v}^D\left(|Y^n_s|+|Z^n_s|+L_g\|\kappa\|\|U^n_s\|\right)  \Big) \\
&\le &  \left|\Gamma_{r,v}(s)\right|^2+|\rho^D_{r,v}|^2\left(|Y^n_s|^2+|Z^n_s|^2+\|U^n_s\|^2\right)  +\frac{a(s)}{2}\left|D_{r,v}Y^{n}_s\right|^2+\frac{1}{2}\left|D_{r,v}Z^n_s\right|^2 \\
&& + \frac{1}{2} \| D_{r,v} U^n_s \|^2 +\left(1+2a(s)+(1\vee L_g\|\kappa\|)^2\left(3+2b(s)^2\right)\right)\left| D_{r,v}Y^{n+1}_s\right|^2 .
\tion 

Choosing 
$\beta=2+3a+(1\vee L_g\|\kappa\|)^2(3+2b^2)$
leads to

\equa
&& \les \E\int_r^T e^{\int_0^s\beta(\tau)d\tau}(1+a(s))\left|D_{r,v}Y^{n+1}_s\right|^2 ds  +\E\int_{{]r,T]}\times\R}e^{\int_0^s\beta(\tau)d\tau}\left|D_{r,v}\underline{Z}^{n+1}_{s,x}\right|^2\m(ds,dx)\\
&\leq& e^{\int_0^T\beta(s)ds}\E\left|D_{r,v}\xi\right|^2+\E\int_r^T e^{\int_0^s\beta(\tau)d\tau} \left|\Gamma_{r,v}(s)\right|^2ds\\
&&+|\rho_{r,v}^D|^2 \, \E\int_r^Te^{\int_0^s\beta(\tau)d\tau}|\left(|Y^n_s|^2+|Z^n_s|^2+\|U^n_s\|^2\right)ds\\
&&+\frac{1}{2}\Biggl(\E\int_r^T \!e^{\int_0^s\beta(\tau)d\tau}(1+a(s))\left|D_{r,v}Y^{n}_s\right|^2ds
\left.
+\E\int_{{]r,T]}\times\R}\!\! e^{\int_0^s\beta(\tau)d\tau} \left|D_{r,v}\underline{Z}^{n}_{s,x}\right|^2 \m(ds,dx)\right).
\tion
Since $\|(Y^n,Z^n,U^n)\|_{L_2(W)\times L_2(W)\times L_2(\tilde{N})}$ converges, we have that 
$$  C_{\sup} :=  \sup_{l\geq 0}\E\int_0^T\left(|Y^l_s|^2+|Z^l_s|^2+\|U^l_s\|^2\right)ds < \infty.$$
Finally, we use  \eqref{D-of-y-n+1} to extend the integrals w.r.t.~$ds$ onto $[0,T],$  and conclude by an elementary  elementary recursion inequality (see Lemma A.1 in \cite{GeissSteinII}) that
\equal  \label{boundedYn}
&& \les \int_0^Te^{\int_0^s\beta(\tau)d\tau} \| (1+a(s)) D Y^{n}_s \|_{L_2(\m\otimes \mathbb{P})}^2ds 
+ \int_{{[0,T]}\times\R}e^{\int_0^T\beta(\tau)d\tau}\|D \underline{Z}^{n}_{s,x}\|_{L_2(\m\otimes \mathbb{P})}^2\m(ds,dx) \non\\
 && \le c_{\beta}\left( \| D\xi\|_{L_2(\mathbb{P} \otimes \m)}^2+\|\Gamma\|_{L_2(\lambda\otimes\mathbb{P} \otimes \m)}^2\right )+ c_{\beta} \|\rho^D\|^2_{L_2(\m)} C_{\sup} \quad  \text{ for all } n \in \N.
\tionl
\smallskip

{\bf Step 3.}\\
We now prove that 
\equal \label{diffnorm}
\left\|\mathcal{Y}-D Y^{n+1}\right\|^2_{\mathrm{L}^2(\mathbb{P}\otimes\lambda\otimes\m)}+\left\|\underline{\mathcal{Z}}-D \underline{Z}^{n+1}\right\|^2_{\mathrm{L}^2(\mathbb{P}\otimes(\m)^{\otimes 2})} \to 0, \quad n \to \infty.
\tionl

To show  \eqref {diffnorm}, one can repeat the above computations, now  for the difference $\mathcal{Y}^{r,v}_t-D_{r,v} Y^{n+1}_t,$ to get
\equal \label{expecteditoddiff}
&& \les \E\int_r^T e^{\int_0^s \beta(\tau)d\tau}\beta(s)(\mathcal{Y}^{r,v}_s-D_{r,v}Y^{n+1}_s)^2ds
+\E\int_{{]r,T]}\times\R}e^{\int_0^s \beta(\tau)d\tau}(\underline{\mathcal{Z}}^{r,v}_{s,x}-D_{r,v}\underline{Z}^{n+1}_{s,x})^2ds\mu(dx) \non \\
&\le& \E\int_r^T e^{ \int_0^s \beta(\tau)d\tau}2\left|\mathcal{Y}^{r,v}_s-D_{r,v}Y^{n+1}_s\right|\non\\
&&\quad \times\left|F_{r,v}(s,\mathcal{Y}^{r,v}_s, \mathcal{Z}^{r,v}_s,  \mathcal{U}^{r,v} _s)-D_{r,v}f(s,  Y^n_s,  Z^n_s, G(s,U^n_s )) \right| ds.
\tionl
We first consider the case $v=0.$ By using Lipschitz properties of $f$ (which also imply the boundedness of the partial derivatives) and \eqref{derivative} it follows that
\equa
&&  \les \left|F_{r,0}(s,\mathcal{Y}^{r,0}_s, \mathcal{Z}^{r,0}_s, \mathcal{U}^{r,0}_s)-D_{r,0}f(s,  Y^n_s,  Z^n_s, G(s,U^n_s) )\right| \phantom{ \bigcup } \non \\
&\le& 
  a(s)\left| \mathcal{Y}^{r,0}_s   -D_{r,0}Y^n_s \right| 
   + b(s)\left( \big | \mathcal{Z}^{r,0}_s   -D_{r,0}Z^n_s \big | + L_g\|\kappa\| \big\| \mathcal{U}^{r,0}_s - D_{r,0}U^n_s \big \| \right) \non \\
&&+ \big | \mathcal{Y}^{r,0}_s  \big |   \,  \big |\partial_y  f(s,Y_s, Z_s, G(s,  U_s))   -\partial_y f(s,  Y^n_s,  Z^n_s, G(s,U^n_s )) \big |  \non   + \lambda_n(r,s),
\tion
 where
 \equa
 \lambda_n(r,s) &:= & \big (\big | (D_{r,0}f)(s,Y_s, Z_s, G(s, U_s))  - (D_{r,0}f) (s,  Y^n_s,  Z^n_s, G(s,U^n_s )) \big |\non\\
 &&\wedge  \bigg( 2 \Gamma_{r,0}(s)+{\rho}^D_{r,0}\Big(|Y^n_s|+|Y_s|+  |Z^n_s|+|Z_s|
 +L_g\|\kappa\|  \Big (\left\|U^n_s\right\| + \left \|U_s\right\|\Big)\bigg)\non \\
&& +\big| \mathcal{Z}^{r,0}_s \big| \,  \big |\partial_z f_g(s,Y_s, Z_s,  U_s)    -\partial_z f_g(s,  Y^n_s,  Z^n_s, U^n_s )\big|  \non \phantom{\bigcup}\\
&& +\big\| \mathcal{U}^{r,0}_s \big \| \bigg(  \big |\partial_u  f(s,Y_s, Z_s, G(s,U_s))   -\partial_u f(s,  Y^n_s,  Z^n_s, G(s,U^n_s)) \big| \,  L_g  \|\kappa\| \non \phantom{\bigcup}\\
&& \quad \quad +  b(s)  \left \|  |g'(s,U_s) - g'(s,U^n_s)| \kappa \right \|  \bigg).
\tion
Thus, using Young's inequality again, and the fact that $|\partial_yf(X,s,\eta)|\leq a(s)$, we estimate
$$2\left|\mathcal{Y}^{r,0}_s-D_{r,0}Y^{n+1}_s\right|\non\\
\left|F_{r,0}(s,\mathcal{Y}^{r,0}_s, \mathcal{Z}^{r,0}_s,  \mathcal{U}^{r,0} _s)-D_{r,0}f(s,  Y^n_s,  Z^n_s, G(s, U^n_s)) \right|$$
by
\begin{align*}
&\left(1+4a(s)+2(1\vee L_g^2\|\kappa\| )^2b(s)^2\right)\left|\mathcal{Y}^{r,0}_s-D_{r,0}Y^{n+1}_s\right|^2\\
&+\lambda_n(r,s)^2+\big | \mathcal{Y}^{r,0}_s  \big |^2   \,  \big |\partial_y  f(s,Y_s, Z_s, G(s,U_s))   -\partial_y f(s,  Y^n_s,  Z^n_s, G(s,U^n_s))\big|   \\
&+\frac{1}{2}\left(a(s)\left|\mathcal{Y}^{r,0}_s-D_{r,0}Y^n_s\right|^2+\big | \mathcal{Z}^{r,0}_s   -D_{r,0}Z^n_s \big |^2 + \big\| \mathcal{U}^{r,0}_s - D_{r,0}U^n_s \big \|^2\right).
\end{align*}
We notice that  $C(s) :=  1+4a(s)+2(1\vee L_g^2\|\kappa\| )^2b(s)^2$ from the expression above is in $L_1([0,T])$.  For the case  $v=0,$  we set 

\equa
 &&\delta_n:=\E\int_0^T\int_0^T e^{\int_0^s\beta(\tau)d\tau}  \Big(\big | \mathcal{Y}^{r,0}_s  \big |^2   \big |\partial_y  f(s,Y_s, Z_s, G(s, U_s))   -\partial_y f(s,  Y^n_s,  Z^n_s, G(s,U^n_s ))\big| \non\\
 &&\quad\quad\quad\quad\quad\quad\quad\quad\quad\quad\quad\quad +\lambda_n(r,s)^2 \Big) drds
 \tion
and are now in the position to infer relation \eqref{final-inequality} below  for the r.h.s. of  \eqref{expecteditoddiff} by using  the above inequalities. The fact that  
 $$\delta_n     \to 0  \,\,\text{ for } \,\, n\to \infty
 $$
 can be seen by Vitali's convergence theorem  taking into consideration that $\beta \in L_1([0,T])$ and  $\E\sup_{t\in{[0,T]}}\big | \mathcal{Y}^{r,0}_t  \big |^2<\infty;$ 
that $( Y^n, Z^n, U^n  )$ converges to  $(Y, Z, U)$ in   $L_2(W)\times L_2(W) \times L_2(\tilde N);$ 
 that $  \partial_y f $ is continuous and bounded by  the function $a.$  For the convergence of the integral part containing $\lambda_n(r,s)^2$ we need additionally that $\eta \mapsto (D_{r,0}f)(s,\eta )$  is continuous (which follows from  (${\A}_f$) $\ref{f-Mall-diff-Lip}  $), that $\partial_zf,\partial_uf$ are continuous and bounded by $b,$ and  that $g'$ is continuous and bounded by $L_g.$ Thanks to the minimum in its first term, $\lambda_n(r,s)^2$ is uniformly integrable in $n.$ \smallskip \\
Now we continue with the case $v\neq 0.$  
We  first  realise  that  for a given $\vare >0$ we may choose  $\alpha >0$ small enough such that  for all $n\ge 1$
 \begin{align}\label{smalleps}
& \E\!\!\int_0^T\int_0^T \int_{\{|v| < \al\}} \!\!\!\! e^{ \int_0^T\beta(\tau)d\tau}\left|\mathcal{Y}^{r,v}_s-D_{r,v}Y^{n+1}_s\right|\non\\
 &\times\left|F_{r,v}(s,\mathcal{Y}^{r,v}_s\!, \mathcal{Z}^{r,v}_s\!,  \mathcal{U}^{r,v} _s)-D_{r,v}f(s,  Y^n_s,  Z^n_s, G(s,U^n_s )) \right| \nu(dv)drds < \vare.
\end{align}
This is because  from  \eqref{generator-estimate}, $({\A}_f$) $\ref{f-Mall-diff-bounded}$  and      \eqref{Dgenerator-estimate} we have that
\begin{align*}
             \left|F_{r,v}(s,\mathcal{Y}^{r,v}_s, \mathcal{Z}^{r,v}_s ,\mathcal{U}^{r,v}_s)\right|  \le & \Gamma_{r,v}(s) + a(s)|\mathcal{Y}^{r,v}_s|+b(s)(|\mathcal{Z}^{r,v}_s| + L_g\|\kappa\|\|\mathcal{U}^{r,v}_s \|   )\\
&+ \rho^D_{r,v}(|Y_s|+|Z_s|+L_g\|\kappa\|\|U_s\||)
\end{align*}
and 
\begin{align*}
    \left|D_{r,v}f(s,Y^n_s,Z^n_s, G(s,U^n_s))\right|
           \le & \Gamma_{r,v}(s) + a(s)| D_{r,v} Y^n_s|+ b(s)(|D_{r,v}Z^n_s| + L_g\|\kappa\|_{L_2(\nu)}\|D_{r,v}{U^n}_s \| )\\
 &+ \rho^D_{r,v}(|Y^n_s|+|Z^n_s|+L_g\|\kappa\|\|U^n_s\||),
\end{align*}
 holds. Then, Young's inequality and inequality \eqref{boundedYn} imply the boundedness of the integral in \eqref{smalleps}.

On the set $\{ |v| \ge \alpha\}$ we  use the Lipschitz properties (${\A}_f$) $\ref{f-Lip} $ and  (${\A}_f$) $\ref{g-condition}$    to get the estimate

\equa
&&\les  \left|F_{r,v}(s,\mathcal{Y}^{r,v}_s, \mathcal{Z}^{r,v}_s ,\mathcal{U}^{r,v}_s)-D_{r,v}f(s,Y^n_s,Z^n_s,G(s,U^n_s))\right| \\
&\le& \big |f\left((X+v \one_{[r,T]}),s,Y_s+\mathcal{Y}^{r,v}_s, Z_s+\mathcal{Z}^{r,v}_s, G(s, U_s + \mathcal{U}^{r,v}_s )
 \right) \\
 && \quad - f\left((X+v\one_{[r,T]}),s,Y^n_s + D_{r,v} Y^n_s, Z^n_s+D_{r,v} Z^n_s , G(s,U^n_s + D_{r,v} U^n_s) \right)  \big | \\
&& + \big | f\left(X,s,Y_s, Z_s, G(s, U_s)\right)  -f\left(X,s,Y^n_s, Z^n_s, G(s, U^n_s)\right) \big |\\
&\le&  a(s)  | \mathcal{Y}^{r,v}_s -  D_{r,v} Y^n_s | +b(s)\big [ |\mathcal{Z}^{r,v}_s-D_{r,v}Z^n_s|+ L_g\|\kappa\|\| \mathcal{U}^{r,v}_s - D_{r,v} U^n_s \| \big]\\
&& + 2a(s) |Y_s- Y^n_s|+b(s)\big [ |Z_s -Z^n_s|+ L_g\|\kappa\| \| U_s- U^n_s \| ) \big ]  .
\tion
This helps us to estimate an integrated version of the r.h.s. of \eqref{expecteditoddiff} for any  $ n\in \N$: Using Young's inequality once again, we arrive for $v=0$ and $v \neq 0$   at \begin{align} \label{final-inequality}
& \les \E\!\int_r^T \!\!\int_{[0,T]\times\R} \! e^{ \int_0^T\beta(\tau)d\tau}\left|\mathcal{Y}^{r,v}_s-D_{r,v}Y^{n+1}_s\right|\non\\
 &\times\left|F_{r,v}(s,\mathcal{Y}^{r,v}_s\!, \mathcal{Z}^{r,v}_s\!,  \mathcal{U}^{r,v} _s)-D_{r,v}f(s,  Y^n_s,  Z^n_s, G(s,U^n_s )) \right| \m(dr,dv)  ds \non  \\
 &\le  \frac{1}{2}\E \int_r^T e^{ \int_0^T\beta(\tau)d\tau} \big (a(s)\| \mathcal{Y}_s-  D Y_s^n \|_{L_2(\m)}^2  +  \|\underline{\mathcal{Z}}_{s,.}-D \underline{Z}^n_{s,.}\|_{L_2(\m\otimes \mu)}^2  \big )ds \non \\
 & +  \nu(\{ |v| \ge \al \}) \E \int_r^T e^{ \int_0^T\beta(\tau)d\tau} \big(a(s)|Y_s- Y^n_s|^2+ \|\underline{Z}_{s,.}-\underline{Z}^n_{s,.}\|^2_{L_2(\m)} \big)ds   \non  \\
&+ \delta_n+      \varepsilon+\E \int_r^T e^{ \int_0^T\beta(\tau)d\tau} C(s)\| \mathcal{Y}_s-  D Y_s^{n+1} \|_{L_2(\m)}^2ds.
\end{align}

Choosing  $\beta=1+a(s)+C(s)$ in  \eqref{expecteditoddiff} and applying \eqref{final-inequality}   leads to
\equa
&&\les \left\|\sqrt{(1+a)}(\mathcal{Y}-D Y^{n+1})\right\|^2_{\mathrm{L}^2(\mathbb{P}\otimes\lambda\otimes\m)}+\left\|\underline{\mathcal{Z}}-D \underline{Z}^{n+1}\right\|^2_{\mathrm{L}^2(\mathbb{P}\otimes(\m)^{\otimes 2})}\\
&\le &\!\! \vare\! + \!C_n\! +\!\frac{1}{2}\left( \left\|\sqrt{(1+a)}(\mathcal{Y}-D Y^{n})\right\|^2_{\mathrm{L}^2(\mathbb{P}\otimes\lambda\otimes\m)}\!\!\!\!
  +\left\|\underline{\mathcal{Z}}-D \underline{Z}^{n}\right\|^2_{\mathrm{L}^2(\mathbb{P}\otimes(\m)^{\otimes 2})}\!\right)
\tion
with  
\begin{align*}
C_n&=C_n(\alpha)\\
&=\delta_n+\nu(\{ |v| \ge \al \})\E \int_r^T\!\! e^{ \int_0^T\beta(\tau)d\tau} \big(a(s)|Y_s- Y^n_s|^2\!+\!\|\underline{Z}_{s,.}-\underline{Z}^n_{s,.}\|^2_{L_2(\m)} \big)ds
\end{align*}
 tending to zero if $n\to \infty$ for any fixed $\alpha>0.$ We now use the recursion inequality from  (\cite[Lemma A.1]{GeissSteinII}) and end up with
\begin{equation*}
\begin{split}
&\limsup_{n\to\infty}\left( \left\|\sqrt{(1+a)}(\mathcal{Y}-D Y^{n})\right\|^2_{\mathrm{L}^2(\mathbb{P}\otimes\lambda\otimes\m)}\!\!\!+\left\|\underline{\mathcal{Z}}-D \underline{Z}^{n}\right\|^2_{\mathrm{L}^2(\mathbb{P}\otimes(\m)^{\otimes 2})}\right) \le 2 \vare.\\
\end{split}
\end{equation*}
This implies \eqref{Y-and-Z-in-D12} since $\left\|\sqrt{(1+a)}\mathcal{Y}\right\|_{\mathrm{L}^2(\mathbb{P}\otimes\lambda\otimes\m)}\!\!\!\!\!<\!\infty$ and $\left\|\sqrt{(1+a)}\ \cdot\ \right\|_{\mathrm{L}^2(\mathbb{P}\otimes\lambda\otimes\m)}\geq \|\cdot\|_{\mathrm{L}^2(\mathbb{P}\otimes\lambda\otimes\m)}$. Hence we can take the Malliavin derivative $D_{r,v}$ of  \eqref{bsde} and get  \eqref{diffrep} for $0\le r\le t\le T$  as well as
\equal  \label{the-other-D-equation}
0= D_{r,v}\xi+\int_r^T F_{r,v}\left(s,D_{r,v}Y_s, D_{r,v}Z_s,  D_{r,v} U_s \right)ds - \underline{Z}_{r,v}-\!\!\int_{{]r,T]}\times\R} D_{r,v} \underline{Z}_{s,x}M(ds,dx),
\tionl
for $0\le t  < r \le T.$ By the same reasoning as for  $D_{r,v}Y^n$ we may conclude that  the RHS of \eqref{diffrep} has a c\`adl\`ag version which we take for
$D_{r,v}Y.$\bigskip

\ref{gimmel} This assertion we get  by comparing \eqref{U-V-equation} and  \eqref{diffrep}  because of the  uniqueness of $(\mathcal{Y},\mathcal{Z},\mathcal{U}).$ \bigskip

\ref{dalet}  For the discussion on the measurability of  $\lim_{t \searrow r}D_{r,v}Y_t$ w.r.t.~$(r,v,\om)$  which is needed to  take the {\ch conditional expectation $\E[\,\cdot\, |\ftn_{r-}]$} we refer the reader to  the proof of \cite[Theorem 4.4]{GeissSteinII}.  The assertion follows then from comparing  \eqref{diffrep}  with \eqref{the-other-D-equation} and the uniqueness of solutions. {\ch If the  predictable projections $ {\phantom{\int}}^p\left(\left(D_{r,0}Y_r\right)_{r\in{[0,T]}}\right) $  \,\, and 
${\phantom{\int}}^p\left(\left(D_{r,v}Y_r\right)_{r\in{[0,T]}, v\in\R_0}\right)$  exist,  this has the benefit that one has jointly measurable processes which are unique up to indistinguishability.}
\hfill  $\square$

\bibliographystyle{plain}

\end{document}